\documentclass[11pt]{article}
\usepackage{amsfonts,amssymb,amsmath,amsthm}

\usepackage{bm}
\usepackage{palatino}
\usepackage{mathpazo}
\usepackage{inconsolata}
\usepackage[hidelinks]{hyperref}

\usepackage[dvipsnames]{xcolor}
\usepackage{tikz}
\usetikzlibrary{arrows.meta} 
\usetikzlibrary{decorations.markings}
\usetikzlibrary{automata, positioning}
\usetikzlibrary{decorations.pathmorphing}
\usepackage{tabularray}
\UseTblrLibrary{booktabs}

\tikzset{
  CircleTip/.tip={Circle[open,fill=OliveGreen,length=5pt]},
  SquareTip/.tip={Square[open,fill=BrickRed,length=4.5pt]},
  DiamondTip/.tip={Turned Square[open,fill=MidnightBlue,length=6pt]},
  t1/.style={
    OliveGreen,
    ultra thick
  },
  circ/.style={
  	shorten >=-2.5pt, shorten <=-2.5pt,
    CircleTip-CircleTip,
  },
  t2/.style={
    BrickRed,
    ultra thick
  },
  sq/.style={
	shorten >=-2.25pt, shorten <=-2.25pt, 
	SquareTip-SquareTip,
  },
  t3/.style={
    MidnightBlue,
    ultra thick
  },
  diam/.style={
	shorten >=-3pt, shorten <=-3pt, 
	DiamondTip-DiamondTip,
  },
  a1/.style={
  	decoration={
	    markings,
    	mark=at position 0.63 with {\arrow[line width=1pt]{Computer Modern Rightarrow}}
    },
    postaction={decorate},
  },
  a2/.style={
  	decoration={
	    markings,
    	mark=at position 0.71 with {\arrow[line width=1pt]{Computer Modern Rightarrow[]Computer Modern Rightarrow[]}}
    },
    postaction={decorate},
  },
  a3/.style={
  	decoration={
	    markings,
    	mark=at position 0.79 with {\arrow[line width=1pt]{Computer Modern Rightarrow[]Computer Modern Rightarrow[]Computer Modern Rightarrow[]}}
    },
    postaction={decorate}
  }
}
\newcommand{\graydots}[1]{
	\foreach \x/\y in {#1}{
		\draw[gray, fill=white] (\x,\y) circle (1.9pt);
	}
}

\usepackage{enumitem}

\begingroup
    \makeatletter
    \@for\theoremstyle:=definition,remark,plain\do{%
        \expandafter\g@addto@macro\csname th@\theoremstyle\endcsname{%
            \addtolength\thm@preskip\parskip
            }%
        }
\endgroup

\usepackage[margin=1in]{geometry}

\theoremstyle{definition}
\newtheorem{theorem}{Theorem}

\newtheorem{definition}{Definition}

\numberwithin{theorem}{section}
\numberwithin{question}{section}
\numberwithin{definition}{section}
\numberwithin{example}{section}
\numberwithin{remark}{section}

\newcommand*\patchAmsMathEnvironmentForLineno[1]{%
  \expandafter\let\csname old#1\expandafter\endcsname\csname #1\endcsname
  \expandafter\let\csname oldend#1\expandafter\endcsname\csname end#1\endcsname
  \renewenvironment{#1}%
     {\linenomath\csname old#1\endcsname}%
     {\csname oldend#1\endcsname\endlinenomath}}%
\newcommand*\patchBothAmsMathEnvironmentsForLineno[1]{%
  \patchAmsMathEnvironmentForLineno{#1}%
  \patchAmsMathEnvironmentForLineno{#1*}}%
\AtBeginDocument{%
\patchBothAmsMathEnvironmentsForLineno{equation}%
\patchBothAmsMathEnvironmentsForLineno{align}%
\patchBothAmsMathEnvironmentsForLineno{flalign}%
\patchBothAmsMathEnvironmentsForLineno{alignat}%
\patchBothAmsMathEnvironmentsForLineno{gather}%
\patchBothAmsMathEnvironmentsForLineno{multline}%
}

\usepackage{parskip}
\usepackage{lineno}
\usepackage[small]{titlesec}

\usepackage[square,comma,numbers,sort&compress]{natbib}

\usepackage{color}
\newcommand{\ds}{\displaystyle}
\newcommand{\N}{\mathbb{N}}
\newcommand{\Z}{\mathbb{Z}}
\newcommand{\GG}{\mathcal{G}}
\renewcommand{\SS}{\mathcal{S}}
\newcommand{\WW}{\mathcal{W}}

\DeclareMathOperator{\wt}{wt}

\renewenvironment{abstract}{
	\begin{list}{}%
	{\setlength{\rightmargin}{1in}%
	\setlength{\leftmargin}{1in}}%
	\item[]\ignorespaces\begin{small}}%
	{\end{small}\unskip\end{list}%
}

\newpagestyle{main}[\small]{
	\headrule
	\sethead[\usepage][][]
	{}{}{\usepage}
}

\title{Exactly-Solvable Self-Trapping Lattice Walks. II. Lattices of Arbitrary Height.}

\usepackage{makecell}
\author{
	\begin{tabular}{m{2.5in}m{2.5in}}
		\makecell{
			Jay Pantone\\
			\small Department of Mathematical\\
			\small and Statistical Sciences\\
			\small Marquette University\\
			\small Milwaukee, WI, USA\\
			\small \texttt{jay.pantone@marquette.edu}
		}&
		\makecell{
			Alexander R. Klotz\\
			\small Department of Physics and Astronomy\\
			\small California State University, Long Beach\\
			\small Long Beach, CA, USA\\
			\small \texttt{alex.klotz@csulb.edu}
		}
		\\\ \\
		\multicolumn{2}{c}{\makecell{
			Everett Sullivan\\
			\small Department of Mathematics\\
			\small Clayton State University\\
			\small Morrow, GA, USA\\
			\small \texttt{EverettSullivan@clayton.edu}
		}}
	\end{tabular}
}

\titleformat{\section}{\large}{\thesection.}{1em}{}
\date{}

\usepackage{arydshln}
\usepackage{makecell}

\begin{document}
\maketitle

\begin{abstract}
	A growing self-avoiding walk (GSAW) is a walk on a graph that is directed, does not visit the same vertex twice, and has a trapped endpoint. We show that the generating function enumerating GSAWs on a half-infinite strip of finite height is rational, and we give a procedure to construct a combinatorial finite state machine that allows one to compute this generating function. We then modify this procedure to compute generating functions for GSAWs under two probabilistic models. We derive the mean trapping lengths for GSAWs in strips of height 2 to 5 to gain insight into the empirically known square lattice result of 71 steps. Finally, we prove that the generating functions for Greek key tours (GSAWs on a finite grid that visit every vertex) on a half-infinite strip of fixed height are also rational, allowing us to resolve several conjectures.
\end{abstract}

\tableofcontents


\section{Growing Self-Avoiding Walks}

A \emph{self-avoiding walk} (SAW) is a walk on a graph that does not visit the same vertex twice~\cite{madras:self-avoiding-walk-book}. Self-avoiding walks are used to model polymer chains in good solvent conditions, and exact enumerations of SAWs provide insight into scaling exponents that can be measured in polymer experiments~\cite{clisby:saws-lace}. Typically, each SAW of a given length is treated as equally probable such that the average size of a polymer chain, which is able to adopt these equally likely configurations through thermal fluctuations, may be computed. In contrast, a \emph{growing self-avoiding walk} (GSAW) begins at the origin and grows on a lattice one step at a time, adding adjacent previously unoccupied sites to the end of the walk according to a probability distribution. When the walk reaches a site that has no unoccupied neighbors, it is said to become trapped and the walk terminates. The GSAW may be used to model polymers whose polymerization rate is much greater than their relaxation time. The most famous GSAW result used Monte Carlo simulations to show that on a square lattice, the mean number of steps taken by a GSAW before it gets trapped is approximately 71~\cite{hemmer:saw-average-71}, with a positively skewed distribution. Other work has computed trapping statistics for other lattices in two and three dimensions~\cite{renner:saw-heteropolymers}, incorporated the effects of self-attraction to study poor-solvent polymer interactions~\cite{hooper:trapping-saws}, and examined an off-lattice model of a semiflexible polymer GSAW~\cite{arakawa:root-mean-squares}. Beyond polymers, recent experiments have shown that droplets on a cold hydrophobic surface can transport themselves while depleting solvent from the surface, creating a trail in their wake that they avoid, in a physical realization of GSAW-like behavior~\cite{lin:emergent-collection-motion}. While the 71-step trapping length is a well-known result, it is purely empirical. The primary goal of this manuscript, along with Part 1~\cite{klotz:gsaws1}, is to develop an exactly solvable model of GSAW trapping to provide insight into this value.

This manuscript develops an approach to exactly computing the mean trapping lengths of GSAWs in finite-height square lattices. In addition to computing the trapping length, it also determines the mean \emph{displacement} of the walk, the difference between the maximum and minimum reached in the $x$-direction. In the study of confined polymers, particularly with regard to genomic-length DNA in nanochannels, this parameter can be readily measured by fluorescence imaging of the molecules and predicted from polymer scaling models~\cite{reisner:dna}. This parameter is of particular importance to nanochannel-based genomic mapping technologies, for which knowing the relationship between spatial and genetic coordinates is crucial~\cite{sheats:measurements}. An exactly solvable random walk model, the one-dimensional Domb-Joyce model, has been used to predict the mean extension of a DNA molecule in a nanochannel in the ``extended de Gennes'' regime~\cite{werner:confined}. Here, we develop another set of exactly-solved results for the displacement of GSAWs.

A \emph{grid graph} is a graph whose vertices are the integer points on the Cartesian plane, or a subset thereof, with edges between any two points whose distance is $1$. We denote by $\GG_{S}$ the grid graph whose vertices are the subset $S$ of integer coordinates. We use $\GG_h$ as short-hand for $\GG_{\N \times \{0, \ldots, h-1\}}$, the half-infinite grid graph with height $h$. Figure~\ref{figure:gg-examples} shows $\GG_{\{0, 1, 2, 3\} \times \{0, 1, 2\}}$, a grid graph with $12$ vertices, and $\GG_4 = \GG_{\N \times \{0, 1, 2, 3\}}$.

GSAWs can be studied with a probabilistic viewpoint. Suppose that at each step in a walk, the next vertex visited is chosen from among the unoccupied neighbors of the current endpoint according to some probability distribution. This implies that for each GSAW $W$ with starting vertex $v$ there is a probability $p(W)$ that a walk starting at $v$ will evolve into $W$ before becoming trapped. For example, if we consider a walk in which we choose among the unoccupied neighbors with uniform probability, then for the walk $W$ in Figure~\ref{figure:first-prob-example}, we have $p(W) = 1/576$.

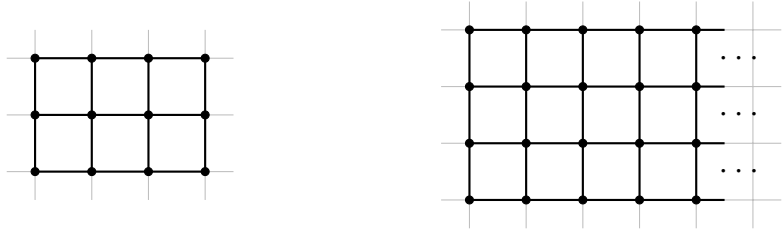
\begin{figure}
	\begin{center}
		\begin{tikzpicture}[scale=0.75, baseline=(current bounding box.center)]
			\draw[thin, lightgray] (-0.5, -0.5) grid (3.5, 2.5);
			\foreach \x in {0, 1, 2, 3}{
				\foreach \y in {0, 1, 2}{
					\draw[black, fill] (\x, \y) circle (2pt);
				}
			}
			\draw[thick] (0,0) -- (3,0);
			\draw[thick] (0,1) -- (3,1);
			\draw[thick] (0,2) -- (3,2);
			\draw[thick] (0,0) -- (0,2);
			\draw[thick] (1,0) -- (1,2);
			\draw[thick] (2,0) -- (2,2);
			\draw[thick] (3,0) -- (3,2);
		\end{tikzpicture}
		\hspace{1in}
		\begin{tikzpicture}[scale=0.75, baseline=(current bounding box.center)]
			\draw[thin, lightgray] (-0.5, -0.5) grid (5.5, 3.5);
			\foreach \x in {0, 1, 2, 3, 4}{
				\foreach \y in {0, 1, 2, 3}{
					\draw[black, fill] (\x, \y) circle (2pt);
				}
			}
			\draw[thick] (0,0) -- (4.5,0);
			\draw[thick] (0,1) -- (4.5,1);
			\draw[thick] (0,2) -- (4.5,2);
			\draw[thick] (0,3) -- (4.5,3);
			\draw[thick] (0,0) -- (0,3);
			\draw[thick] (1,0) -- (1,3);
			\draw[thick] (2,0) -- (2,3);
			\draw[thick] (3,0) -- (3,3);
			\draw[thick] (4,0) -- (4,3);
			\node at (4.8, 2.5) {\large $\cdots$};
			\node at (4.8, 1.5) {\large $\cdots$};
			\node at (4.8, 0.5) {\large $\cdots$};
		\end{tikzpicture}
	\end{center}
	\caption{On the left, the finite grid graph $\GG_{\{0, \ldots, 3\} \times \{0, \ldots, 2\}}$. On the right, the half-infinite grid graph $\GG_4$.}
	\label{figure:gg-examples}
\end{figure}

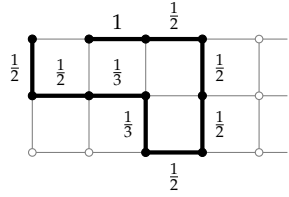
\begin{figure}
	\begin{center}
		\begin{tikzpicture}[
			scale=0.75,
			baseline=(current bounding box.center),
		]
			\draw[gray] (0,0) grid (4.5,2);
			\foreach \x in {0,1,2,3,4} { \foreach \y in {0,1,2} {
				\draw[gray, fill=white] (\x,\y) circle (1.9pt);
			}}
			
			\def\points{(0,2),(0,1),(1,1),(2,1),(2,0),(3,0),(3,1),(3,2),(2,2),(1,2)}

			\draw[ultra thick] (0,2)
				\foreach \point in \points {
					-- \point
				};
			\foreach \point in \points {
				\draw[black, fill] \point circle (2pt);
			}
			\node[left] at (0, 1.5) {\scriptsize $\frac{1}{2}$};
			\node[above] at (0.5, 1) {\scriptsize $\frac{1}{2}$};
			\node[above] at (1.5, 1) {\scriptsize $\frac{1}{3}$};
			\node[left] at (2, 0.5) {\scriptsize $\frac{1}{3}$};
			\node[below] at (2.5, 0) {\scriptsize $\frac{1}{2}$};
			\node[right] at (3, 0.5) {\scriptsize $\frac{1}{2}$};
			\node[right] at (3, 1.5) {\scriptsize $\frac{1}{2}$};
			\node[above] at (2.5, 2) {\scriptsize $\frac{1}{2}$};
			\node[above] at (1.5, 2) {\scriptsize $1$};
  		\end{tikzpicture}
	\end{center}
	\caption{A GSAW on $\GG_3$ that starts at $(0,2)$ and ends trapped at $(1,2)$. The probability of each edge is shown. The probability that this GSAW occurs on $\GG_3$ is the product of the edge probabilities, $1/576$.}
	\label{figure:first-prob-example}
\end{figure}

In the prequel to this work, Klotz and Sullivan~\cite{klotz:gsaws1} studied GSAWs on the doubly-infinite height two grid graph $\GG_{\Z \times \{0,1\}}$ and a similar triangular lattice of height two. For any graph $G$ we use $W_v(G)$ to denote the set of GSAWs on $G$ starting at vertex $v$, and for a walk $W$ we use $|W|$ to denote the length of $W$ (the number of edges) and when $\GG$ is a grid graph we use $\|W\|$ to denote the displacement of $W$ (the difference between the largest and smallest $x$-coordinates of vertices involved in $W$).

Using direct combinatorial methods, Klotz and Sullivan computed, for each of the two different height-two grid graphs $G$ and for two different probability measures, the generating function
\[
	f^G_v(x, y) = \sum_{W \in W_v(G)} p(W)x^{|W|}y^{\|W\|}
\]
that counts walks in $W_v(G)$ according to their length and displacement and weighted by their probability. Observe that we may also write
\[
	f^G_v(x, y) = \ds\sum_{n,k \geq 0} p_{n,k}x^ny^k
\]
where $p_{n,k}$ is the sum of the probabilities of all GSAWs on $G$ with starting vertex $v$ that have $n$ edges and displacement $k$. From each generating function, Klotz and Sullivan computed a number of properties such as the expected length and expected displacement of a GSAW.

In this work, we show that one can construct finite state machines that generate GSAWs on half-infinite grid graphs of any height, i.e., any graph $\GG_h$, and we provide a Python implementation to do so. As a result, the generating function $f_v^G$ for any such $G$ and any vertex $v$ is guaranteed to be rational, and can be algorithmically computed in finite time, allowing us to significantly extend the results of Klotz and Sullivan to these graphs\footnote{The methods we describe can be extended to the fully-infinite case $\GG_{\Z \times \{0, \ldots, h-1\}}$ in a straight-forward but difficult manner.}. With a small modification, we also use the finite state machine approach to answer questions posed by Johnston~\cite{johnston:greek-key-blog} about SAWs on finite grid graphs that visit every vertex.

There is substantial literature on the topic of describing various types of self-avoiding walks through finite state machines and transfer matrices. We discuss this prior literature, as well as summarize our main results, in Section~\ref{section:prev-lit}. We further discuss there why our approach differs significantly from those in the literature. Following this, we show how to construct finite state machines for the non-probabilistic version of this problem in Section~\ref{section:fsm-non-prob}. We use these constructions in Section~\ref{section:non-prob-results} to derive enumerative results about GSAWs on half-infinite grid graphs of any height. Sections~\ref{section:fsm-prob} and~\ref{section:prob-results} demonstrate how to adapt the construction to incorporate each of the two probability distributions considered by Klotz and Sullivan~\cite{klotz:gsaws1} and then again derive enumerative results about GSAWs considered under these distributions. Section~\ref{section:full-GSAWs} uses a small modification of our approach to answer questions of Jensen~\cite{jensen:compact-saws} and Johnston~\cite{johnston:greek-key-blog}. 

We conclude in Section~\ref{section:conclusion} by posing questions about GSAWs on other grid graphs and discussing how a more general finite state machine approach may be possible. Finally, Appendix~\ref{subsection:monte-carlo} reports on the results of Monte Carlo simulations of GSAWs on various grid graphs.

\section{Prior Literature and Summary of Main Results}
\label{section:prev-lit}
In this section we summarize the main results of our paper in the context of previous literature on the topic. In this work, we develop a system of finite state machines that allows us to compute generating functions for the trapping length and displacement distribution of growing self-avoiding walks in semi-infinite square lattice strips. From these generating functions we are able to exactly compute the mean trapping length of GSAWs in strips of height 2, 3, 4, and 5, for which the results are 13, $\approx19.32$, $\approx22.98$, and $\approx26.52$ respectively. From these exact results we can extrapolate that mean trapping length in the quarter-infinite plane is approximately 45.8, which is close to the Monte Carlo estimate of 45.4. We believe this is a significant step towards understanding the classic 71 result of Hemmer and Hemmer \cite{hemmer:saw-average-71}. We are also able to compute the expected displacement of trapped walks, which are 7,  $\approx7.75$, $\approx7.83$, and $\approx8.08$ over the same range of heights. This height-displacement relation may be thought of as analogous to exact solutions of nanoconfined polymer problems \cite{werner:confined}.

The tools we developed to arrive at these numbers can be used to answer other related questions regarding confined SAWs and GSAWs. Recently, Laforge, Mikayelyan, and Senet~\cite{laforge:trapping} extended the work of Hooper and Klotz~\cite{hooper:trapping-saws} to study repulsive walks on square and honeycomb lattices. They found that the trapping length reaches a plateau with increasing repulsion strength (the bias in the step probability distribution against sites adjacent to those already occupied) for both lattices. We have developed an ``energetic model'' which incorporates a cost or benefit of nearest-neighbor contacts into the walk probability distribution. In the limit of repulsive walks ($C=0$ in our notation), we can exactly compute the plateau of the trapping length. For lattices larger than 2 we can exactly predict the trapping at $C=0$, finding it is $3867/112\approx34.5$ for $h=3$, and approximately $33.8$ for $h=4$ and $38.9$ for $h=5$, compared to the unconfined square lattice value of approximately $178.5$, or $107.9\pm0.3$ in the quarter-infinite plane. 

We are also able to use our finite state machines to investigate the enumerations of what have variously been called Hamiltonian walks, Greek key tours, or maximal length GSAWs: walks that visit every lattice site up to a certain width. We are able to exactly enumerate these up to strips of height 8. Previously, Jensen~\cite{jensen:compact-saws} computed enumerations of Hamiltonian paths in finite-width strips, which are equivalent to the Greek Keys we study in Section~\ref{section:full-GSAWs}. Jensen is able to arrive at numerical estimates of the exponential growth rates, which we are able to calculate exactly from our generating functions.

Going beyond our specific results, our work adds to the body of work on enumeration of self-avoiding walks, and in some cases we can provide exact solution where only numerical estimates existed previously. There has been significant work focusing on the enumeration of self-avoiding walks in confining geometries. Much of this work uses finite state machine and transfer matrix methods, as do we, but two key differences are that our focus is on probabilistic models and that most earlier literature is not centered on GSAWs or trapping. As a result, our work requires a more complicated notion of ``state''---what we call a width-2 window for the uniform model and a width-4 window for the energetic model---in order for the probabilities of each transition to be computed. 

Having made these caveats, we will now describe the similar work that predates ours. Wall, Seitz, Chin, and Mandel~\cite{wall:saws} computed the average span of SAWs in strips of width two and three using straightforward transfer matrix methods, and Klein~\cite{klein:saws} extended this to study walks in strips bound on both ends. Klein also computed connective constants for the enumerations of confined walks that are the same as those we find for GSAWs in strips of the same width. This provides support to the conjecture that the universality class of GSAWs is the same as that of SAWS, which is difficult to determine from examining the asymptotic behavior of large unconfined walks. A number of studies that focus on enumeration of unconfined walks use walks in strips as a transitional step. These include Conway, Enting, and Guttmann~\cite{conway:saws} who computed the number of walks up to length 39, and arrived at enumerations of confined walks as a transitional step, as well as Clisby and Jensen~\cite{clisby:new-transfer}. Most recently, Dangovski and Lalov~\cite{dangovski:saws} arrived numerically at similar connective constants that we derive exactly.


\section{The Non-Probabilistic Finite State Machine Construction}
\label{section:fsm-non-prob}

\subsection{The Frames of a GSAW}
\label{subsection:frames}

Figure~\ref{figure:frame-ex-1} shows a GSAW that starts at the vertex $(0, 4)$, has length $65$, and has displacement $14$. The finite state machine construction that we will describe in this section builds GSAWs in a left-to-right manner by considering at each step a \emph{frame} of width 2 containing only the edges of the GSAW that are fully contained in that frame. The leftmost frame, the one containing columns $0$ and $1$, of the GSAW from Figure~\ref{figure:frame-ex-1} is shown on the left in Figure~\ref{figure:frame-ex-2}. It contains the seven edges from the GSAW that have both endpoints with $x$-coordinates $0$ or $1$, and we call this \emph{Frame 0}. Frame 1 contains the seven edges whose endpoints both have $x$-coordinate $1$ or $2$, Frame 2 contains the four edges whose endpoints both have $x$-coordinates $2$ or $3$, and so on. 

\begin{figure}
	\begin{center}
		\begin{tikzpicture}[
			scale=0.75,
			baseline=(current bounding box.center),
		]
			\draw[thin, gray] (0,0) grid (15.5,4);
			\foreach \x in {0,...,15} { \foreach \y in {0,1,2,3,4} {
				\draw[gray, fill=white] (\x,\y) circle (1.9pt);
			}}
			
			\def\points{(0,4),(1,4),(2,4),(3,4),(4,4),(4,3),(4,2),(4,1),(3,1),(2,1),(1,1),(1,2),(2,2),(2,3),(1,3),(0,3),(0,2),(0,1),(0,0),(1,0),(2,0),(3,0),(4,0),(5,0),(6,0),(7,0),(7,1),(7,2),(7,3),(6,3),(6,2),(5,2),(5,3),(5,4),(6,4),(7,4),(8,4),(9,4),(9,3),(9,2),(10,2),(11,2),(12,2),(13,2),(13,3),(12,3),(11,3),(11,4),(12,4),(13,4),(14,4),(14,3),(14,2),(14,1),(14,0),(13,0),(12,0),(12,1),(11,1),(10,1),(10,0),(9,0),(9,1),(8,1),(8,2),(8,3)}

			\foreach \point in \points {
				\draw[black, fill] \point circle (2.3pt);
			}
			
			\draw[line width=2pt] (0,4) -- (1,4) -- (2,4) -- (3,4) -- (4,4) -- (4,3) -- (4,2) -- (4,1) -- (3,1) -- (2,1) -- (1,1) -- (1,2) -- (2,2) -- (2,3) -- (1,3) -- (0,3) -- (0,2) -- (0,1) -- (0,0) -- (1,0) -- (2,0) -- (3,0) -- (4,0) -- (5,0) -- (6,0) -- (7,0) -- (7,1) -- (7,2) -- (7,3) -- (6,3) -- (6,2) -- (5,2) -- (5,3) -- (5,4) -- (6,4) -- (7,4) -- (8,4) -- (9,4) -- (9,3) -- (9,2) -- (10,2) -- (11,2) -- (12,2) -- (13,2) -- (13,3) -- (12,3) -- (11,3) -- (11,4) -- (12,4) -- (13,4) -- (14,4) -- (14,3) -- (14,2) -- (14,1) -- (14,0) -- (13,0) -- (12,0) -- (12,1) -- (11,1) -- (10,1) -- (10,0) -- (9,0) -- (9,1) -- (8,1) -- (8,2) -- (8,3);
			
  		\end{tikzpicture}
	\end{center}
	\caption{A GSAW on $\GG_5$ with length $65$ and displacement $14$.}
	\label{figure:frame-ex-1}
\end{figure}
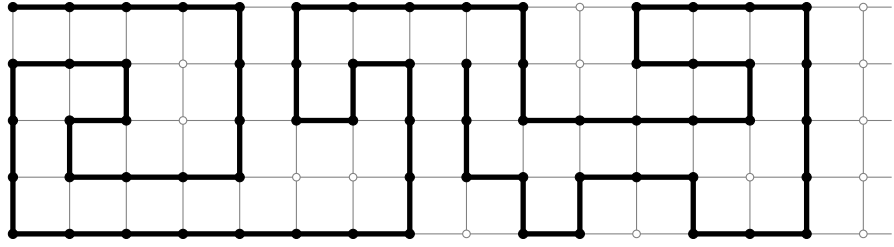

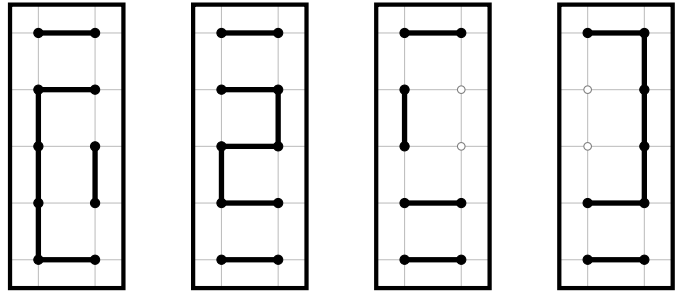
\begin{figure}
	\begin{center}
		\hfill
		\begin{tikzpicture}[
			scale=0.75,
			baseline=(current bounding box.center),
		]
			\draw[thin, lightgray] (-0.5,-0.5) grid (1.5,4.5);
			\foreach \x in {0,1} { \foreach \y in {0,1,2,3,4} {
				\draw[gray, fill=white] (\x,\y) circle (1.9pt);
			}}
			\draw[ultra thick] (-0.5,-0.5) rectangle (1.5,4.5);
			
			\def\points{(0,0),(0,1),(0,2),(0,3),(0,4),(1,0),(1,1),(1,2),(1,3),(1,4)};

			\foreach \point in \points {
				\draw[black, fill] \point circle (2.3pt);
			}

			\draw [line width=2pt] (0,4) -- (1,4);
			\draw [line width=2pt] (1,3) -- (0,3) -- (0,0) -- (1,0);		
			\draw [line width=2pt] (1,1) -- (1,2);
  		\end{tikzpicture}
  		\quad\quad
		\begin{tikzpicture}[
			scale=0.75,
			baseline=(current bounding box.center),
		]
			\draw[thin, lightgray] (-0.5,-0.5) grid (1.5,4.5);
			\foreach \x in {0,1} { \foreach \y in {0,1,2,3,4} {
				\draw[gray, fill=white] (\x,\y) circle (1.9pt);
			}}
			\draw[ultra thick] (-0.5,-0.5) rectangle (1.5,4.5);
			
			\def\points{(0,0),(0,1),(0,2),(0,3),(0,4),(1,0),(1,1),(1,2),(1,3),(1,4)};

			\foreach \point in \points {
				\draw[black, fill] \point circle (2.3pt);
			}

			\draw [line width=2pt] (0,4) -- (1,4);
			\draw [line width=2pt] (0,3) -- (1,3) -- (1,2) -- (0,2) -- (0,1) -- (1,1);
			\draw [line width=2pt] (0,0) -- (1,0);
  		\end{tikzpicture}
  		\quad\quad
		\begin{tikzpicture}[
			scale=0.75,
			baseline=(current bounding box.center),
		]
			\draw[thin, lightgray] (-0.5,-0.5) grid (1.5,4.5);
			\foreach \x in {0,1} { \foreach \y in {0,1,2,3,4} {
				\draw[gray, fill=white] (\x,\y) circle (1.9pt);
			}}
			\draw[ultra thick] (-0.5,-0.5) rectangle (1.5,4.5);
			
			\def\points{(0,0),(0,1),(0,2),(0,3),(0,4),(1,0),(1,1),(1,4)};

			\foreach \point in \points {
				\draw[black, fill] \point circle (2.3pt);
			}

			\draw [line width=2pt] (0,4) -- (1,4);
			\draw [line width=2pt] (0,2) -- (0,3);
			\draw [line width=2pt] (0,1) -- (1,1);
			\draw [line width=2pt] (0,0) -- (1,0);
					
  		\end{tikzpicture}
  		\quad\quad
		\begin{tikzpicture}[
			scale=0.75,
			baseline=(current bounding box.center),
		]
			\draw[thin, lightgray] (-0.5,-0.5) grid (1.5,4.5);
			\foreach \x in {0,1} { \foreach \y in {0,1,2,3,4} {
				\draw[gray, fill=white] (\x,\y) circle (1.9pt);
			}}
			\draw[ultra thick] (-0.5,-0.5) rectangle (1.5,4.5);
			
			\def\points{(0,0),(0,1),(0,4),(1,0),(1,1),(1,2),(1,3),(1,4)};

			\foreach \point in \points {
				\draw[black, fill] \point circle (2.3pt);
			}

			\draw [line width=2pt] (0,0) -- (1,0);
			\draw [line width=2pt] (0,1) -- (1,1) -- (1,4) -- (0,4);
					
  		\end{tikzpicture}
  		\hfill\ {}
	\end{center}
	\caption{The leftmost four frames of the GSAW in Figure~\ref{figure:frame-ex-1}.}
	\label{figure:frame-ex-2}
\end{figure}

We will say much more about finite state machines later, but the goal of the finite state machine approach is to discover a finite number of building blocks from which we may build all GSAWs by arranging these blocks in orders that we decide are valid. In order to determine if a certain building block can be used next, one may look only at the most recently used block, not any prior blocks. The frames that we have described above will serve as the base for our building blocks, but they are not sufficient on their own. Consider for example the four frames shown on the left in Figure~\ref{figure:frame-ex-3}. When concatenated together, they define the graph on the right in Figure~\ref{figure:frame-ex-3} that is clearly not a GSAW. The issue is that although each frame aligns with the next, they do not contain enough information for us to realize that when concatenating the rightmost frame to the first three, we will create a graph that contains a cycle.

\begin{figure}
	\begin{center}
		\hfill
		\begin{tikzpicture}[
			scale=0.75,
			baseline=(current bounding box.center),
		]
			\draw[thin, lightgray] (-0.5,-0.5) grid (1.5,4.5);
			\foreach \x in {0,1} { \foreach \y in {0,1,2,3,4} {
				\draw[gray, fill=white] (\x,\y) circle (1.9pt);
			}}
			\draw[ultra thick] (-0.5,-0.5) rectangle (1.5,4.5);
			
			\def\points{(0,1),(0,2),(0,3),(0,4),(1,0),(1,1),(1,2),(1,3),(1,4)};

			\foreach \point in \points {
				\draw[black, fill] \point circle (2.3pt);
			}

			\draw [line width=2pt] (0,4) -- (0,3) -- (1,3) -- (1,4);
			\draw [line width=2pt] (1,0) -- (1,1) -- (0,1) -- (0,2) -- (1,2);			
  		\end{tikzpicture}
  		\quad
  		\begin{tikzpicture}[
			scale=0.75,
			baseline=(current bounding box.center),
		]
			\draw[thin, lightgray] (-0.5,-0.5) grid (1.5,4.5);
			\foreach \x in {0,1} { \foreach \y in {0,1,2,3,4} {
				\draw[gray, fill=white] (\x,\y) circle (1.9pt);
			}}
			\draw[ultra thick] (-0.5,-0.5) rectangle (1.5,4.5);
			
			\def\points{(0,0),(0,1),(0,2),(0,3),(0,4),(1,0),(1,2),(1,4)};

			\foreach \point in \points {
				\draw[black, fill] \point circle (2.3pt);
			}

			\draw [line width=2pt] (0,1) -- (0,0) -- (1,0);
			\draw [line width=2pt] (1,2) -- (0,2);
			\draw [line width=2pt] (0,3) -- (0,4) -- (1,4);
  		\end{tikzpicture}
  		\quad
  		\begin{tikzpicture}[
			scale=0.75,
			baseline=(current bounding box.center),
		]
			\draw[thin, lightgray] (-0.5,-0.5) grid (1.5,4.5);
			\foreach \x in {0,1} { \foreach \y in {0,1,2,3,4} {
				\draw[gray, fill=white] (\x,\y) circle (1.9pt);
			}}
			\draw[ultra thick] (-0.5,-0.5) rectangle (1.5,4.5);
			
			\def\points{(0,0),(0,2),(0,4),(1,0),(1,2),(1,4)};

			\foreach \point in \points {
				\draw[black, fill] \point circle (2.3pt);
			}

			\draw [line width=2pt] (0,0) -- (1,0);
			\draw [line width=2pt] (0,2) -- (1,2);
			\draw [line width=2pt] (0,4) -- (1,4);
  		\end{tikzpicture}
  		\quad
  		\begin{tikzpicture}[
			scale=0.75,
			baseline=(current bounding box.center),
		]
			\draw[thin, lightgray] (-0.5,-0.5) grid (1.5,4.5);
			\foreach \x in {0,1} { \foreach \y in {0,1,2,3,4} {
				\draw[gray, fill=white] (\x,\y) circle (1.9pt);
			}}
			\draw[ultra thick] (-0.5,-0.5) rectangle (1.5,4.5);
			
			\def\points{(0,0),(0,2),(0,4),(1,0),(1,1),(1,2),(1,4)};

			\foreach \point in \points {
				\draw[black, fill] \point circle (2.3pt);
			}

			\draw [line width=2pt] (0,4) -- (1,4);
			\draw [line width=2pt] (0,0) -- (1,0) -- (1,1) -- (1,2) -- (0,2);			
  		\end{tikzpicture}
  		\hfill\hfill
		\begin{tikzpicture}[
			scale=0.75,
			baseline=(current bounding box.center),
		]
			\draw[gray] (0,0) grid (5.5,4);
			\foreach \x in {0,1,2,3,4,5} { \foreach \y in {0,1,2,3,4} {
				\draw[gray, fill=white] (\x,\y) circle (1.9pt);
			}}
			
			\def\points{(0,1),(0,2),(0,3),(0,4),(1,0),(1,1),(1,2),(1,3),(1,4),(2,0),(2,2),(2,4),(3,0),(3,2),(3,4),(4,0),(4,1),(4,2),(4,4)}

			\foreach \point in \points {
				\draw[black, fill] \point circle (2pt);
			}
			
			\draw[ultra thick] (0,4) -- (0,3) -- (1,3) -- (1,4) -- (4,4);
			\draw[ultra thick] (0,2) -- (4,2) -- (4,0) -- (1,0) -- (1,1) -- (0,1) -- cycle;
			
  		\end{tikzpicture}
  		\hfill\ {}
	\end{center}
	\caption{On the left, four frames such that each consecutive pair is compatible. On the right, the non-GSAW graph you get when you join them all together.}
	\label{figure:frame-ex-3}
\end{figure}
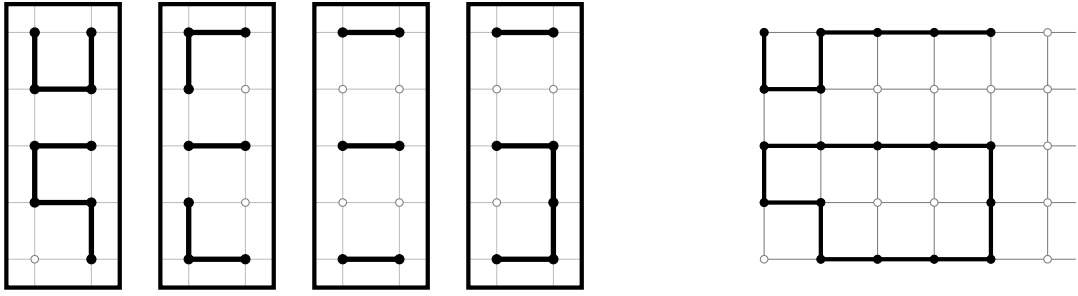

To remedy this, we will define a version of a frame that encodes not just which edges are present, but also the order in which they appear in an actual GSAW that uses the frame, as well as which edges are already connected by edges from previous frames.

Before giving a formal definition, we illustrate this idea using again the GSAW in Figure~\ref{figure:frame-ex-1}. Frame $0$ will consist of the seven edges that it did before, but each connected group of edges will be directed and we will also designate the order in which the connected components appear. The leftmost picture in Figure~\ref{figure:frame-ex-5} shows this new Frame $0$. The arrows denote the direction of the edges in each part of the GSAW, and we have used colors and endpoint shapes to denote the order in which the parts appear: the green part with circular endpoints comes first, the red part with square endpoints comes second, and the blue part with diamond endpoints comes third. 

In Frame $1$, the red and blue parts from Frame $0$ have now joined together. Put another way, if we looked only at the edges from the GSAW whose endpoints have $x$-coordinates $0$, $1$, or $2$, we would see only two connected components. So, even though Frame $1$ has three separate connected components, we color two of them red (with square endpoints) to denote that these parts were already connected by edges to the left of the frame. In order to still remember the order in which the red parts are traversed, we write one directional arrow on the first part, and two directional arrows on the second part. The remainder of Figure~\ref{figure:frame-ex-5} shows Frames $2$ and $3$.

\begin{figure}
	\begin{center}
		\begin{tikzpicture}[
			scale=0.75,
			baseline=(current bounding box.center),
		]
			\draw[lightgray] (-0.5,-0.5) grid (1.5,4.5);
			\graydots{}
			\draw[ultra thick] (-0.5,-0.5) rectangle (1.5,4.5);
			
		    \draw[t1, a1, circ] (0,4) -- (1,4); 
		    \draw[t2, a1, sq] (1,1) -- (1,2);
		    \draw[t3, a1] (0,2) -- (0,1);
			\draw[t3, diam] (1,3) -- (0,3) -- (0,2) -- (0,1) -- (0,0) -- (1,0);
		\end{tikzpicture}
		\quad\quad
		\begin{tikzpicture}[
			scale=0.75,
			baseline=(current bounding box.center),
		]
			\draw[lightgray] (-0.5,-0.5) grid (1.5,4.5);
			\graydots{}
			\draw[ultra thick] (-0.5,-0.5) rectangle (1.5,4.5);
			
		    \draw[t1, a1, circ] (0,4) -- (1,4); 
		    \draw[t2, a1] (0,2) -- (1,2);
		    \draw[t2, a2, sq] (0,0) -- (1,0);
		    \draw[t2, sq] (0,3) -- (1,3) -- (1,2) -- (0,2) -- (0,1) -- (1,1);
		\end{tikzpicture}
		\quad\quad
		\begin{tikzpicture}[
			scale=0.75,
			baseline=(current bounding box.center),
		]
			\draw[lightgray] (-0.5,-0.5) grid (1.5,4.5);
			\graydots{1/2,1/3}
			\draw[ultra thick] (-0.5,-0.5) rectangle (1.5,4.5);
			
		    \draw[t1, a1, circ] (0,4) -- (1,4); 
		    \draw[t2, a2, sq] (0,2) -- (0,3);
		    \draw[t2, a1, sq] (1,1) -- (0,1);
		    \draw[t2, a3,sq] (0,0) -- (1,0);
		\end{tikzpicture}
		\quad\quad
		\begin{tikzpicture}[
			scale=0.75,
			baseline=(current bounding box.center),
		]
			\draw[lightgray] (-0.5,-0.5) grid (1.5,4.5);
			\graydots{0/2,0/3}
			\draw[ultra thick] (-0.5,-0.5) rectangle (1.5,4.5);
			
		    \draw[t1, a1] (1,3) -- (1,2); 
		    \draw[t1, a2, circ] (0,0) -- (1,0); 
		    \draw[t1, circ] (0,4) -- (1,4) -- (1,1) -- (0,1);	    
		\end{tikzpicture}
	\end{center}
	\caption{The leftmost four frames of the GSAW in Figure~\ref{figure:frame-ex-1}, augmented with information about the order in which the components appear.}
	\label{figure:frame-ex-5}
\end{figure}
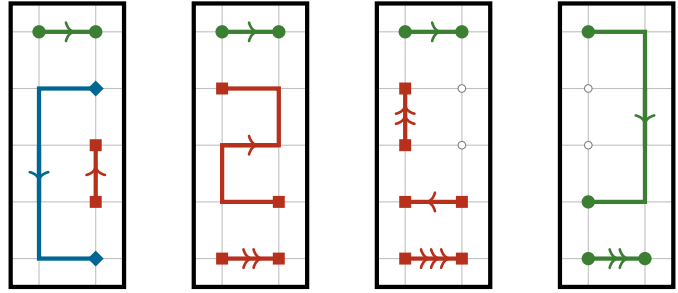

With this information, we can determine whether a given frame can be concatenated to a GSAW that is being built from left-to-right using only information from the most recent frame. The ordering of the groups of edges and the knowledge of whether two groups have already been connected is sufficient to guarantee that the final set of edges produced is in fact a GSAW. 

In order to give a formal definition of a frame we first need to introduce two intermediate concepts, paths and segments.

\begin{definition}
	A \emph{path} in a graph is a nonempty sequence of points $P_1, \ldots, P_\ell$ that forms a sequence of directed edges with consistent directionality: each internal vertex in the path is the end of one edge and the start of another. Note that a path is allowed to consist of a single point, in which case there is no directionality.
\end{definition}
\begin{definition}
	A \emph{segment} is an ordered list of paths that share no vertices in common.
\end{definition}
\begin{definition}
	A \emph{frame} is an ordered list of segments that share no vertices in common.
\end{definition}

Pictorially, we use a different color for each segment of a frame, e.g., the leftmost frame in Figure~\ref{figure:frame-ex-5} contains three segments in the order green, red, and blue, and each of these segments contains only a single path. We denote the order of multiple paths in a single segment by the number of arrowheads drawn, e.g., the third frame from the left in Figure~\ref{figure:frame-ex-5} has two segments, the second of which has three paths, each with a single edge. The first path is $(1,1) \to (0,1)$, the second path is $(0,2) \to (0,3)$, and the third path is $(0,0) \to (1,0)$. 

\subsection{Finite State Machines}
\label{subsection:finite-state-machines}

A finite state machine is a construction from the field of formal language theory that has found applications in many other areas, including in combinatorics. In typical usage, one has a finite set of symbols (or \emph{letters}) $\Sigma$ called the \emph{alphabet}, a finite set of \emph{states} $S$, a designated \emph{start state} $s \in S$, a set of \emph{accepting states} $A \subseteq S$, and a transition function $\delta : S \times \Sigma \to S$ that maps each (state, letter) pair to a state. A \emph{word} is a sequence of letters from the alphabet. Any given word defines a sequence of states that the finite state machines moves through as it reads one letter at a time, and this word is considered \emph{accepted} if the final state is an accepting state, and \emph{rejected} otherwise. See Sipser's classic book~\cite{sipser:theory-of-computation} for more details.

\begin{figure}
    \centering
    \begin{minipage}{0.45\textwidth}
		\begin{center}
			\begin{tikzpicture}[shorten >=1pt, shorten <=1pt,node distance=2cm,on grid,auto,initial text={}]
				\node[state,initial,accepting] (1) {1};
				\node[state,accepting] (2) [above right=of 1, xshift=30pt, yshift=-5pt] {2};
				\node[state] (3) [below right=of 1, xshift=30pt, yshift=5pt] {3};
				
				\path[->]
				(1) edge [bend left=15] node {$b$} (2)
				 edge [loop above] node {$a$} ()
				(2) edge [bend left=15] node {$b$} (1)
				 edge [bend left=15] node {$a$} (3)
				(3) edge [bend left=15] node {$a$} (1)
				 edge [bend left=15] node {$b$} (2);
			\end{tikzpicture}
		\end{center}
		\caption{A traditional finite state machine with three states.}
		\label{figure:fsm-1}
	\end{minipage}
	\hfill
    \begin{minipage}{0.45\textwidth}
        \begin{center}
			\begin{tikzpicture}[shorten >=1pt, shorten <=1pt,node distance=2cm,on grid,auto,initial text={}]
				\node[state,initial,accepting] (1) {1};
				\node[state,accepting] (2) [above right=of 1, xshift=30pt, yshift=-5pt] {2};
				\node[state] (3) [below right=of 1, xshift=30pt, yshift=5pt] {3};
				
				\path[->]
				(1) edge [bend left=15] node[yshift=5pt, xshift=14pt] {$2x$} (2)
				 edge [loop above] node {$x^2$} ()
				(2) edge [bend left=15] node[xshift=-15pt, yshift=-5pt] {$x$} (1)
				 edge [bend left=15] node[yshift=-5pt]  {$x+x^2$} (3)
				 edge [bend right=15, swap] node[yshift=-5pt] {$3x^2$} (3)
				(3) edge [bend left=15] node {$x$} (1);
			\end{tikzpicture}
		\end{center}
		\caption{A combinatorial finite state machine with three states.}
		\label{figure:fsm-2}
    \end{minipage}
\end{figure}

Consider for example the finite state machine depicted in Figure~\ref{figure:fsm-1} which has alphabet $\Sigma = \{a,b\}$, state set $S = \{1, 2, 3\}$, start state $1$ (designated by the incoming arrow), accepting states $\{1, 2\}$ (designated by the double circle), and transition function $\delta$ defined by the arrows, i.e, $\delta(1, a) = 1$, $\delta(1, b) = 2$, $\delta(2, a) = 3$, and so on. When the word $w = ababbbaa$ is input into this finite state machine, it starts in state $1$, stays in state $1$ after reading $w_1 = a$, moves to state $2$ after reading $w_2 = b$, moves to state $3$ after reading $w_3 = a$, etc. The full state sequence is $1, 1, 2, 3, 2, 1, 2, 3, 1$. Since the last state in this sequence is an accepting state, the word $w$ is accepted.

Let $a_n$ be the number of words of length $n$ that are accepted by a given finite state machine. The generating function $\sum_{n \geq 0} a_nx^n$ is guaranteed to be rational, and a technique called the \emph{transfer matrix method} makes it possible to determine it algorithmically; see Stanley~\cite[Section 4.7]{stanley:ec1}. One starts by forming the adjacency matrix $M$ whose rows and columns are indexed by the states $S$ such that $M_{i,j}$ is the number of transitions from state $i$ to state $j$. Let $I$ be the identity matrix with the same dimensions as $M$. Then, the generating function for accepted words is
\[
	\sum_{t \in A} (I - xM)^{-1}_{s, t},
\]
that is, the sum of the entries in $(1 - xM)^{-1}$ in the row corresponding to the start state and in columns corresponding to the accepting states. For the finite state machine in Figure~\ref{figure:fsm-1}, the adjacency matrix is
\[
	M = \left[\begin{array}{ccc}
		1 & 1 & 0\\
		1 & 0 & 1\\
		1 & 1 & 0
	\end{array}\right].
\]
Defining $M' = (I - xM)^{-1}$, we find that the generating function is
\[
	f(x) = M'_{1,1} + M'_{1,2} = \frac{1+x-x^2}{(1+x)(1-2x)} = 1 + 2x + 3x^2 + 7x^3 + \cdots.
\]

In many combinatorial applications, including in this work, the alphabet is actually irrelevant, and one really wants only to think of the finite state machine as a directed graph (allowing parallel edges and loops) in which the states are the vertices and the transitions are the edges. We call this a \emph{combinatorial finite state machine}.

For more descriptive power, we can assign each edge a \emph{weight}. We define the weight of a walk $w$ through a directed graph, denoted $\wt(w)$, to be the product of the weights of the edges it traverses. We can again use the transfer matrix method to compute the generating function for weighted walks starting at a designated start vertex and ending at designated accepting vertices. Consider the directed graph shown in Figure~\ref{figure:fsm-2}.

We form in the same way as before the adjacency matrix $M$, setting the entry in row $i$ and column $j$ to be the sum of the weights of all edges from vertex $i$ to vertex $j$. For the example above, we find
\[
	M = \left[\begin{array}{ccc}
		x^2 & 2x & 0\\
		x & 0 & x+4x^2\\
		x & 0 & 0
	\end{array}\right].
\]
Then, the generating function $f(x) = \sum_{w} \wt(w)$ taken over all walks that start at the start vertex and end at an accepting vertex is the sum of the entries in $(1 - M)^{-1}$ in the row corresponding to the start vertex and the columns corresponding to the accepting vertices. Note that the condition that each nonzero entry of $M$ has no constant term is sufficient, but not necessary, to guarantee the invertibility of $1-M$. In this case,
\[
	f(x) = \frac{1}{1-2x+x^2-4x^3} = 1 + 2x + 3x^2 + 8x^3 + \cdots.
\]
The term $2x$ in this series expansion comes from the walk $1 \xrightarrow{2x} 2$ (weights of edges are shown on top of the arrow). The term $3x^2$ comes from two walks: the first is $1 \xrightarrow{x^2} 1$ which has weight $x^2$, and the second is $1 \xrightarrow{2x} 2 \xrightarrow{x} 1$ which has weight $(2x)(x) = 2x^2$. The term $8x^3$ comes from three walks:
\begin{align*}
	1 &\xrightarrow{x^2} 1 \xrightarrow{2x} 2\\
	1 &\xrightarrow{2x} 2 \xrightarrow{x} 1 \xrightarrow{2x} 2\\
	1 &\xrightarrow{2x} 2 \xrightarrow{x+x^2} 3 \xrightarrow{x} 1.	
\end{align*}
The first of these three walks has weight $(x^2)(2x) = 2x^3$, the second has weight $(2x)(x)(2x) = 4x^3$ and the third has weight $(2x)(x+x^2)(x) = 2x^3 + 2x^4$. Thus the total contribution of walks to the $x^3$ term is $8x^3$.

\subsection{Construction of the Directed Graph}
\label{subsection:non-prob-construction}

In this subsection we will describe how to construct, for each grid graph $\GG_h$, a directed graph $D_h$ such that walks in $D_h$ from the start vertex to accepting vertices with weight $x^ny^k$ bijectively correspond to GSAWs with $n$ edges and displacement $k$.\footnote{Note that on the grid graph $\GG_h$, the displacement of a walk is the same as the largest $x$-coordinate of any vertex on the walk.} In other words, the generating function counting walks by weight will be equal to the generating function for GSAWs by length and displacement. In this way, we can guarantee the rationality of the generating functions for GSAWs on these grid graphs and algorithmically compute them.

Fix a height $h \geq 2$. Our frames will therefore have width $2$ and height $h$. In order to describe $D_h$, we need to specify four things: the vertices, the edges and their weights, the start vertex, and the accepting vertices.

For this theoretical description of $D_h$, we can consider the vertices to be all possible frames.\footnote{The vast majority of frames actually cannot arise as part of a GSAW, so when actually computing $D_h$ we will only consider frames that can.} In order to describe the edges of $D_h$, we imagine starting with a frame $F$ of width 2, adding an additional empty column on the right-hand side, extending the segments of $F$ in all possible ways into this new wider frame, and then removing the leftmost column and all edges touching it to obtain a new width $2$ frame $F'$.

\begin{figure}
	\begin{center}
		\begin{tikzpicture}[
			scale=0.75,
			baseline=(current bounding box.center),
		]
			\draw[lightgray] (-0.5,-0.5) grid (1.5,4.5);
			\graydots{}
			\draw[ultra thick] (-0.5,-0.5) rectangle (1.5,4.5);
			
		    \draw[t1, a1, circ] (0,4) -- (1,4); 
		    \draw[t2, a1, sq] (1,1) -- (1,2);
		    \draw[t3, a1] (0,2) -- (0,1);
			\draw[t3, diam] (1,3) -- (0,3) -- (0,2) -- (0,1) -- (0,0) -- (1,0);
			\node at (0.5, -1) {$F_1$};
		\end{tikzpicture}
		\quad\quad
		\begin{tikzpicture}[
			scale=0.75,
			baseline=(current bounding box.center),
		]
			\draw[lightgray] (-0.5,-0.5) grid (2.5,4.5);
			\graydots{2/0,2/1,2/2,2/3,2/4}
			\draw[ultra thick] (-0.5,-0.5) rectangle (2.5,4.5);
			
		    \draw[t1, a1, circ] (0,4) -- (1,4); 
		    \draw[t2, a1, sq] (1,1) -- (1,2);
		    \draw[t3, a1] (0,2) -- (0,1);
			\draw[t3, diam] (1,3) -- (0,3) -- (0,2) -- (0,1) -- (0,0) -- (1,0);
			\node at (1, -1) {$F_2$};
		\end{tikzpicture}
		\quad\quad
		\begin{tikzpicture}[
			scale=0.75,
			baseline=(current bounding box.center),
		]
			\draw[lightgray] (-0.5,-0.5) grid (2.5,4.5);
			\graydots{}
			\draw[ultra thick] (-0.5,-0.5) rectangle (2.5,4.5);
			
		    \draw[t1, a1, circ] (0,4) -- (2,4); 
		    \draw[t2, a1] (1,1) -- (1,2);
		    \draw[ultra thick, BrickRed] (2,1) -- (1,1) -- (1,2) -- (2,2) -- (2,2.5);
		    \draw[t2, sq] (2,1) -- (2,1);
		    \draw[t3, a1] (0,2) -- (0,1);
			\draw[t3] (2,2.5) -- (2,3) -- (1,3) -- (0,3) -- (0,2) -- (0,1) -- (0,0) -- (1,0) -- (2,0);
			\draw[t3, diam] (2,0) -- (2,0);
			\draw[ultra thick, dotted] (1.6, 2.5) -- (2.4, 2.5);
			\node at (1, -1) {$F_3$};
		\end{tikzpicture}
		\quad\quad
		\begin{tikzpicture}[
			scale=0.75,
			baseline=(current bounding box.center),
		]
			\draw[lightgray] (-0.5,-0.5) grid (2.5,4.5);
			\graydots{}
			\draw[ultra thick] (-0.5,-0.5) rectangle (2.5,4.5);
			
		    \draw[t1, a1, circ] (0,4) -- (2,4); 
		    \draw[t2, a1] (1,1) -- (1,2);
		    \draw[t2, sq] (2,1) -- (1,1) -- (1,2) -- (2,2) -- (2,3) -- (1,3) -- (0,3) -- (0,2) -- (0,1) -- (0,0) -- (1,0) -- (2,0);
		    \node at (1, -1) {$F_4$};
		\end{tikzpicture}
		\quad\quad
		\begin{tikzpicture}[
			scale=0.75,
			baseline=(current bounding box.center),
		]
			\draw[lightgray] (-0.5,-0.5) grid (1.5,4.5);
			\graydots{}
			\draw[ultra thick] (-0.5,-0.5) rectangle (1.5,4.5);
			
		    \draw[t1, a1, circ] (0,4) -- (1,4); 
		    \draw[t2, a2, sq] (0,0) -- (1,0);
		    \draw[t2, a1] (0,2) -- (1,2);
		    \draw[t2, sq] (1,1) -- (0,1) -- (0,2) -- (1,2) -- (1,3) -- (0,3);
		    \node at (0.5, -1) {$F_5$};
		\end{tikzpicture}

	\end{center}
	\caption{An example of the frame extension process that defines the edges in $D_h$.}
	\label{figure:frame-transition-1}
\end{figure}
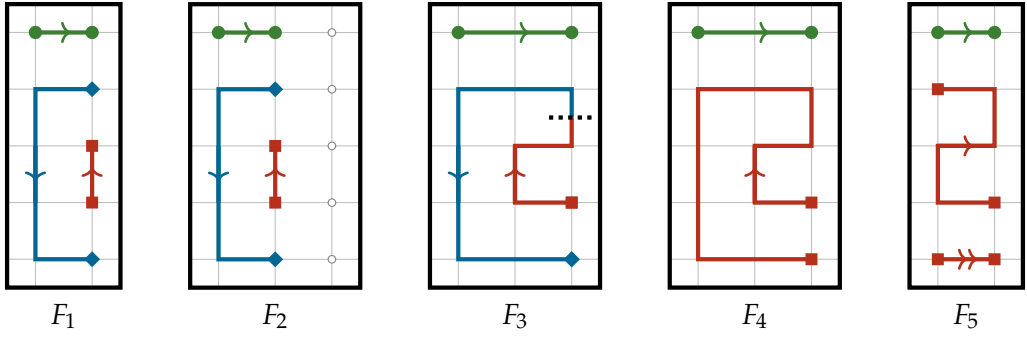

For example, starting with the frame $F_1$ on the left in Figure~\ref{figure:frame-transition-1}, we first add an empty column to obtain $F_2$. We call the first point of the first path of a segment its \emph{start} and the last point of the last path of a segment its \emph{end}. One way that $F_2$ is can be extended, shown as $F_3$, is to advance the end of the first (green) segment from $(1,4)$ to $(2,4)$, advance the start of the second (red) segment backward from $(1,1)$ to $(2,1)$, advance both the end of the second segment and the start of the third (blue) segment rightward and have them merge together, and advance the end of the third segment from $(1,0)$ to $(2,0)$. Because the second and third segment have merged into one (i.e., we have revealed enough of the GSAW we are building to see where they connect), they become a single segment, which is now the second of the two segments in the frame, as shown in $F_4$. Finally, we obtain $F_5$ by trimming off the leftmost column of $F_4$ to return to a width $2$ frame. In doing so, some intermediate edges of the second segment are removed; this converts the second segment from a single path with $11$ edges into two paths, one with $5$ edges and one with $1$ edge. This example shows that in $D_5$, there is an edge from $F_1$ to $F_5$, and the weight of that edge is $x^6y$ because $6$ new edges were added, and the width of the GSAW being built was increased by one.

The general procedure to determine how a frame $F$ may be extended is as follows.
\begin{enumerate}
	\item In any frame that is part of a GSAW, the start and end of every segment must be in the right column with the exception of the start of the first segment (which will never be on the right column) and the end of the last segment (which may be in the right column, but is not required to be). If this is not the case, then a segment will fall out of the frame before connecting to another segment, resulting in a disconnected walk. Therefore, if a frame does not satisfy this condition, it has no outgoing edges.
	\item We can now assume that the frame $F$ satisfies the condition above. By the same logic, in any extension of $F$, the start and end of every segment (except the start of the first and the end of the last) must be extended rightward into the new rightmost column. The end of the last segment may be extended into the new rightmost column if possible, but it is not required (unless a new segment is later added in the last position, as described in Step~\ref{step4}).
	\item Each of these extended ends may then move upward or downward any number of edges in the new rightmost column, as long as they do not overlap. Additionally, the end of each segment $i$ may join up to the start of next segment $i+1$ if doing so would not cause overlap with the vertices of any other segments. If so, we say that these two segments have \emph{merged}. Note that segments may only merge with the next segment in order, but that multiple merges may be possible. For example, it may be possible that segments $i$, $i+1$, and $i+2$ all merge together if the end of segment $i$ joins with the start of segment $i+1$ and the end of segment $i+1$ joins with the start of segment $i+2$.
	\item \label{step4}Finally, new segments may be inserted to the frame. They can occupy any position among the ordered list of segments except for the first, as the first segment must always be the one that started in the top-left corner of $\GG_h$. These new segments consist of a single path that moves upward or downward any number of vertices in the rightmost column, as long as they do not overlap with vertices in other segments. If a new segment is the last segment of the frame, it may consist of only a single point---if this is the case, then when it is extended later, this point is considered the end, and the start is extended rightward. New segments that are not the last segment of the frame must consist of at least two vertices. Lastly, note that if a new segment is made to be the last segment of the frame, then the previous last segment must have had its endpoint extended rightward in an earlier step (i.e., it was not actually optional in that step).
\end{enumerate}

For any width $2$ frame $F$, the procedure above defines a (possibly empty) set $E(F)$ of width $3$ extended frames that $F$ may evolve into. Recall that in a completed GSAW, the endpoint must be trapped; that is, there can be no open neighboring position into which it could have extended further. To ensure that the GSAWs we generate satisfy this property, we must remove from $E(F)$ any frames in which the end of the last segment is not trapped and either the end of the last segment is not in the rightmost column or the last segment is a single point in the rightmost column, either of which indicates that the end of this segment is bound to be the last point in the GSAW. We call the remaining set of extended frames $E'(F)$. Lastly, as in the example in Figure~\ref{figure:frame-transition-1}, each extended frame must have its leftmost column trimmed to turn it back into a width $2$ frame. To trim a frame, we simply remove its leftmost column, and all edges incident to any vertex in the leftmost column. This may result in paths being split into multiple paths; their relative order within the segment is maintained. A segment never splits into multiple segments. 

We call the set of trimmed frames the \emph{neighbors} of $F$ and denote it $N(F)$. These are the frames for which there will be an incoming directed edge originating from $F$, whose weight will be $x^ny$ where $n$ is the number of new edges added.

We take this opportunity to present a comprehensive example of the extension process. Consider the frame $F$ shown in Figure~\ref{figure:comp-ex-1}. The set $E(F)$ of extended frames contains $35$ frames, but many of these fail the condition that the end of the last segment must be trapped unless it is in the rightmost column, including the five examples in Figure~\ref{figure:comp-ex-2}.

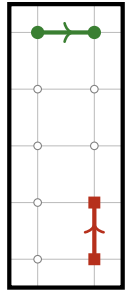
\begin{figure}
	\begin{center}
		\begin{tikzpicture}[
			scale=0.75,
			baseline=(current bounding box.center),
		]
			\draw[lightgray] (-0.5,-0.5) grid (1.5,4.5);
			\graydots{0/0, 0/1, 0/2, 0/3, 1/2, 1/3}
			\draw[ultra thick] (-0.5,-0.5) rectangle (1.5,4.5);
			
		    \draw[t1, a1, circ] (0,4) -- (1,4); 
		    \draw[t2, a1, sq] (1,0) -- (1,1);
		\end{tikzpicture}
	\end{center}
	\caption{A frame with two segments, each containing a single path that has a single edge.}
	\label{figure:comp-ex-1}
\end{figure}

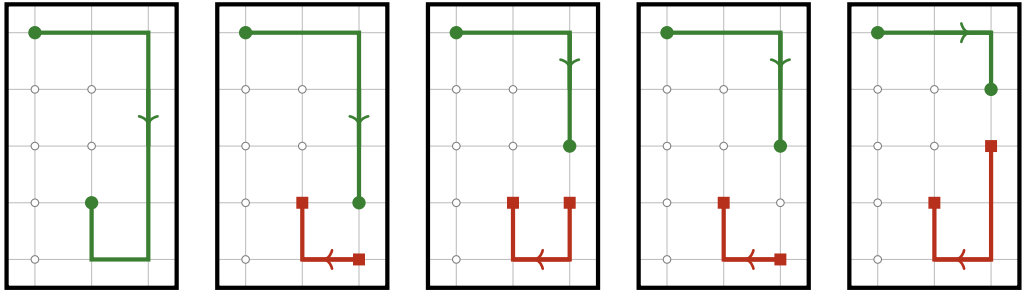
\begin{figure}
	\begin{center}
		\begin{tikzpicture}[
			scale=0.75,
			baseline=(current bounding box.center),
		]
			\draw[lightgray] (-0.5,-0.5) grid (2.5,4.5);
			\graydots{0/0, 0/1, 0/2, 0/3, 1/2, 1/3}
			\draw[ultra thick] (-0.5,-0.5) rectangle (2.5,4.5);
			
		    \draw[t1, circ] (0,4) -- (2,4) -- (2,0) -- (1,0) -- (1,1);
		    \draw[t1, a1] (2,3) -- (2,2); 
		\end{tikzpicture}
		\quad
		\begin{tikzpicture}[
			scale=0.75,
			baseline=(current bounding box.center),
		]
			\draw[lightgray] (-0.5,-0.5) grid (2.5,4.5);
			\graydots{0/0, 0/1, 0/2, 0/3, 1/2, 1/3}
			\draw[ultra thick] (-0.5,-0.5) rectangle (2.5,4.5);
			
		    \draw[t1, circ] (0,4) -- (2,4) -- (2,1);
		    \draw[t1, a1] (2,3) -- (2,2); 
		    		    \draw[t2, sq] (2,0) -- (1,0) -- (1,1);
		    \draw[t2, a1] (2,0) -- (1,0);
		\end{tikzpicture}
		\quad
		\begin{tikzpicture}[
			scale=0.75,
			baseline=(current bounding box.center),
		]
			\draw[lightgray] (-0.5,-0.5) grid (2.5,4.5);
			\graydots{0/0, 0/1, 0/2, 0/3, 1/2, 1/3}
			\draw[ultra thick] (-0.5,-0.5) rectangle (2.5,4.5);
			
		    \draw[t1, circ] (0,4) -- (2,4) -- (2,2);
		    \draw[t1, a1] (2,4) -- (2,3); 
		    \draw[t2, sq] (2,1) -- (2,0) -- (1,0) -- (1,1);
		    \draw[t2, a1] (2,0) -- (1,0);
		\end{tikzpicture}
		\quad
		\begin{tikzpicture}[
			scale=0.75,
			baseline=(current bounding box.center),
		]
			\draw[lightgray] (-0.5,-0.5) grid (2.5,4.5);
			\graydots{0/0, 0/1, 0/2, 0/3, 1/2, 1/3, 2/1}
			\draw[ultra thick] (-0.5,-0.5) rectangle (2.5,4.5);
			
		    \draw[t1, circ] (0,4) -- (2,4) -- (2,2);
		    \draw[t1, a1] (2,4) -- (2,3); 
		    \draw[t2, sq] (2,0) -- (1,0) -- (1,1);
		    \draw[t2, a1] (2,0) -- (1,0);
		\end{tikzpicture}
		\quad
		\begin{tikzpicture}[
			scale=0.75,
			baseline=(current bounding box.center),
		]
			\draw[lightgray] (-0.5,-0.5) grid (2.5,4.5);
			\graydots{0/0, 0/1, 0/2, 0/3, 1/2, 1/3}
			\draw[ultra thick] (-0.5,-0.5) rectangle (2.5,4.5);
			
		    \draw[t1, circ] (0,4) -- (2,4) -- (2,3);
		    \draw[t1, a1] (1,4) -- (2,4); 
		    \draw[t2, sq] (2,2) -- (2,0) -- (1,0) -- (1,1);
		    \draw[t2, a1] (2,0) -- (1,0);
		\end{tikzpicture}

	\end{center}
	\caption{Five extended frames that are discarded because the end point of the last segment is neither in the rightmost column nor trapped.}
	\label{figure:comp-ex-2}
\end{figure}

Once these bad extended frames have been removed, we are left with a set $E'(F)$ with $10$ extended frames. Upon trimming each of these extended frames by removing their leftmost column, we obtain the set $N(F)$ with the $10$ neighbors of $F$, as shown in Figure~\ref{figure:comp-ex-3}.

\begin{figure}
	\begin{center}
		\begin{tikzpicture}[
			scale=0.75,
			baseline=(current bounding box.center),
		]
			\draw[lightgray] (-0.5,-0.5) grid (1.5,4.5);
			\graydots{0/2,0/3}
			\draw[ultra thick] (-0.5,-0.5) rectangle (1.5,4.5);
			
		    \draw[t1,circ] (0,4) -- (1,4) -- (1,2); 
		    \draw[t1,a1] (1,4) -- (1,3);
		    \draw[t2,sq] (1,0) -- (0,0) -- (0,1) -- (1,1);
		    \draw[t2,a1] (0,0) -- (0,1);
		\end{tikzpicture}
		\quad\quad
		\begin{tikzpicture}[
			scale=0.75,
			baseline=(current bounding box.center),
		]
			\draw[lightgray] (-0.5,-0.5) grid (1.5,4.5);
			\graydots{0/2,0/3}
			\draw[ultra thick] (-0.5,-0.5) rectangle (1.5,4.5);
			
		    \draw[t1,circ] (0,4) -- (1,4) -- (1,3); 
		    \draw[t1,a1] (1,4) -- (1,3);
		    \draw[t2,sq] (1,0) -- (0,0) -- (0,1) -- (1,1) -- (1,2);
		    \draw[t2,a1] (0,0) -- (0,1);
		\end{tikzpicture}
		\quad\quad
		\begin{tikzpicture}[
			scale=0.75,
			baseline=(current bounding box.center),
		]
			\draw[lightgray] (-0.5,-0.5) grid (1.5,4.5);
			\graydots{0/2,0/3,1/2}
			\draw[ultra thick] (-0.5,-0.5) rectangle (1.5,4.5);
			
		    \draw[t1,circ] (0,4) -- (1,4) -- (1,3); 
		    \draw[t1,a1] (1,4) -- (1,3);
		    \draw[t2,sq] (1,0) -- (0,0) -- (0,1) -- (1,1);
		    \draw[t2,a1] (0,0) -- (0,1);
		\end{tikzpicture}
		\quad\quad
		\begin{tikzpicture}[
			scale=0.75,
			baseline=(current bounding box.center),
		]
			\draw[lightgray] (-0.5,-0.5) grid (1.5,4.5);
			\graydots{0/2,0/3}
			\draw[ultra thick] (-0.5,-0.5) rectangle (1.5,4.5);
			
		    \draw[t1,circ] (0,4) -- (1,4);
		    \draw[t1,a1] (0,4) -- (1,4);
		    \draw[t2,sq] (1,0) -- (0,0) -- (0,1) -- (1,1) -- (1,3);
		    \draw[t2,a1] (0,0) -- (0,1);
		\end{tikzpicture}
		\quad\quad
		\begin{tikzpicture}[
			scale=0.75,
			baseline=(current bounding box.center),
		]
			\draw[lightgray] (-0.5,-0.5) grid (1.5,4.5);
			\graydots{0/2,0/3,1/3}
			\draw[ultra thick] (-0.5,-0.5) rectangle (1.5,4.5);
			
		    \draw[t1,circ] (0,4) -- (1,4);
		    \draw[t1,a1] (0,4) -- (1,4);
		    \draw[t2,sq] (1,0) -- (0,0) -- (0,1) -- (1,1) -- (1,2);
		    \draw[t2,a1] (0,0) -- (0,1);
		\end{tikzpicture}\\[15pt]

		\begin{tikzpicture}[
			scale=0.75,
			baseline=(current bounding box.center),
		]
			\draw[lightgray] (-0.5,-0.5) grid (1.5,4.5);
			\graydots{0/2,0/3}
			\draw[ultra thick] (-0.5,-0.5) rectangle (1.5,4.5);
			
		    \draw[t1,circ] (0,4) -- (1,4);
		    \draw[t1,a1] (0,4) -- (1,4);
		    \draw[t2,sq] (1,0) -- (0,0) -- (0,1) -- (1,1);
		    \draw[t2,a1] (0,0) -- (0,1);
		\end{tikzpicture}
		\quad\quad
		\begin{tikzpicture}[
			scale=0.75,
			baseline=(current bounding box.center),
		]
			\draw[lightgray] (-0.5,-0.5) grid (1.5,4.5);
			\graydots{0/2,0/3}
			\draw[ultra thick] (-0.5,-0.5) rectangle (1.5,4.5);
			
		    \draw[t1,circ] (0,4) -- (1,4);
		    \draw[t1,a1] (0,4) -- (1,4);
		    \draw[t2,sq] (1,0) -- (0,0) -- (0,1) -- (1,1);
		    \draw[t2,a1] (0,0) -- (0,1);
		    \draw[t3,a1,diam] (1,3) -- (1,2);
		\end{tikzpicture}
		\quad\quad
		\begin{tikzpicture}[
			scale=0.75,
			baseline=(current bounding box.center),
		]
			\draw[lightgray] (-0.5,-0.5) grid (1.5,4.5);
			\graydots{0/2,0/3}
			\draw[ultra thick] (-0.5,-0.5) rectangle (1.5,4.5);
			
		    \draw[t1,circ] (0,4) -- (1,4);
		    \draw[t1,a1] (0,4) -- (1,4);
		    \draw[t2,sq] (1,0) -- (0,0) -- (0,1) -- (1,1);
		    \draw[t2,a1] (0,0) -- (0,1);
		    \draw[t3,a1,diam] (1,2) -- (1,3);
		\end{tikzpicture}
		\quad\quad
		\begin{tikzpicture}[
			scale=0.75,
			baseline=(current bounding box.center),
		]
			\draw[lightgray] (-0.5,-0.5) grid (1.5,4.5);
			\graydots{0/2,0/3}
			\draw[ultra thick] (-0.5,-0.5) rectangle (1.5,4.5);
			
		    \draw[t1,circ] (0,4) -- (1,4);
		    \draw[t1,a1] (0,4) -- (1,4);
		    \draw[t3,diam] (1,0) -- (0,0) -- (0,1) -- (1,1);
		    \draw[t3,a1] (0,0) -- (0,1);
		    \draw[t2,a1,sq] (1,3) -- (1,2);
		\end{tikzpicture}
		\quad\quad
		\begin{tikzpicture}[
			scale=0.75,
			baseline=(current bounding box.center),
		]
			\draw[lightgray] (-0.5,-0.5) grid (1.5,4.5);
			\graydots{0/2,0/3}
			\draw[ultra thick] (-0.5,-0.5) rectangle (1.5,4.5);
			
		    \draw[t1,circ] (0,4) -- (1,4);
		    \draw[t1,a1] (0,4) -- (1,4);
		    \draw[t3,diam] (1,0) -- (0,0) -- (0,1) -- (1,1);
		    \draw[t3,a1] (0,0) -- (0,1);
		    \draw[t2,a1,sq] (1,2) -- (1,3);
		\end{tikzpicture}

	\end{center}
	\caption{The ten neighbors (valid next frames) of the frame show in Figure~\ref{figure:comp-ex-1}.}
	\label{figure:comp-ex-3}
\end{figure}
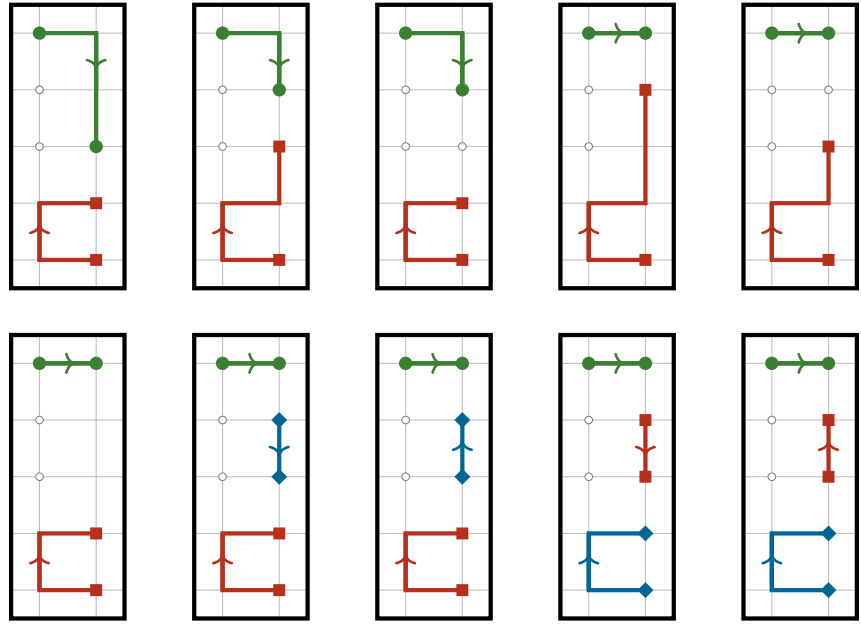

Having described the vertices and edges of $D_h$, it remains only to specify a start vertex and a set of accepting vertices. In this case, we would actually like to allow any vertex of $D_h$ that is a valid first frame of a GSAW to be a start state. To accomplish this, we add to $D_h$ a new vertex $s$ that has edges with weight $x^dy$ to each existing vertex of $D_h$ that represents a valid first frame with $d$ edges. The valid first frames of a GSAW are those in which all segments consist of a single path, only the last segment may consist of a single point, the first segment starts in the top-left corner at the point $(0, h-1)$, the starts of all other segments are in the right column, the ends of all segments (except possibly the last) are in the right column, if the end of the last segment is in the left column then it is trapped, and if the last segment is a single point in the right column then it is also trapped. For example, there are $7$ such first frames for the height $2$ case---these are shown in Figure~\ref{figure:height-2-starts}.

\begin{figure}
    \centering
    \begin{minipage}{0.6\textwidth}
		\begin{center}
			\begin{tikzpicture}[
				scale=0.75,
				baseline=(current bounding box.center),
			]
				\draw[lightgray] (-0.5,-0.5) grid (1.5,1.5);
				\graydots{0/0}
				\draw[ultra thick] (-0.5,-0.5) rectangle (1.5,1.5);
				
			    \draw[t1,circ] (0,1) -- (1,1) -- (1,0);
			    \draw[t1,a1] (1,1) -- (1,0);
			\end{tikzpicture}
			\quad
			\begin{tikzpicture}[
				scale=0.75,
				baseline=(current bounding box.center),
			]
				\draw[lightgray] (-0.5,-0.5) grid (1.5,1.5);
				\graydots{0/0,1/0}
				\draw[ultra thick] (-0.5,-0.5) rectangle (1.5,1.5);
				
			    \draw[t1,circ,a1] (0,1) -- (1,1);
			\end{tikzpicture}
			\quad
			\begin{tikzpicture}[
				scale=0.75,
				baseline=(current bounding box.center),
			]
				\draw[lightgray] (-0.5,-0.5) grid (1.5,1.5);
				\graydots{1/1}
				\draw[ultra thick] (-0.5,-0.5) rectangle (1.5,1.5);
				
			    \draw[t1,circ] (0,1) -- (0,0) -- (1,0);
			    \draw[t1, a1] (0,0) -- (1,0);
			\end{tikzpicture}
			\quad
			\begin{tikzpicture}[
				scale=0.75,
				baseline=(current bounding box.center),
			]
				\draw[lightgray] (-0.5,-0.5) grid (1.5,1.5);
				\graydots{}
				\draw[ultra thick] (-0.5,-0.5) rectangle (1.5,1.5);
				
			    \draw[t1,circ] (0,1) -- (1,1) -- (1,0) -- (0,0);
			    \draw[t1,a1] (1,1) -- (1,0);
			\end{tikzpicture}\ \\[10pt]

			\begin{tikzpicture}[
				scale=0.75,
				baseline=(current bounding box.center),
			]
				\draw[lightgray] (-0.5,-0.5) grid (1.5,1.5);
				\graydots{}
				\draw[ultra thick] (-0.5,-0.5) rectangle (1.5,1.5);
				
			    \draw[t1,circ,a1] (0,1) -- (1,1);
			    \draw[t2,sq,a1] (1,0) -- (0,0);
			\end{tikzpicture}
			\quad
			\begin{tikzpicture}[
				scale=0.75,
				baseline=(current bounding box.center),
			]
				\draw[lightgray] (-0.5,-0.5) grid (1.5,1.5);
				\graydots{}
				\draw[ultra thick] (-0.5,-0.5) rectangle (1.5,1.5);
				
			    \draw[t1,circ] (0,1) -- (0,0) -- (1,0);
			    \draw[t1,a1] (0,0) -- (1,0);
			    \draw[t2,sq] (1,1) -- (1,1);
			\end{tikzpicture}
			\quad
			\begin{tikzpicture}[
				scale=0.75,
				baseline=(current bounding box.center),
			]
				\draw[lightgray] (-0.5,-0.5) grid (1.5,1.5);
				\graydots{}
				\draw[ultra thick] (-0.5,-0.5) rectangle (1.5,1.5);
				
			    \draw[t1,circ] (0,1) -- (0,0) -- (1,0) -- (1,1);
			    \draw[t1,a1] (0,0) -- (1,0);
			\end{tikzpicture}
		\end{center}
		\caption{The seven starting frames for building GSAWs on $\GG_2$. Note that the single red square represents that the second segment contains only a single point, which will be the endpoint of the segment after it is extended further.}
		\label{figure:height-2-starts}
	\end{minipage}
	\hfill
	\begin{minipage}{0.35\textwidth}
		\begin{center}
			\begin{tikzpicture}[
				scale=0.75,
				baseline=(current bounding box.center),
			]
				\draw[lightgray] (-0.5,-0.5) grid (1.5,1.5);
				\graydots{}
				\draw[ultra thick] (-0.5,-0.5) rectangle (1.5,1.5);
				
			    \draw[t1,circ] (0,1) -- (1,1) -- (1,0) -- (0,0);
			    \draw[t1,a1] (1,1) -- (1,0);
			\end{tikzpicture}
			\quad
			\begin{tikzpicture}[
				scale=0.75,
				baseline=(current bounding box.center),
			]
				\draw[lightgray] (-0.5,-0.5) grid (1.5,1.5);
				\graydots{}
				\draw[ultra thick] (-0.5,-0.5) rectangle (1.5,1.5);
				
			    \draw[t1,circ] (0,1) -- (1,1) -- (1,0) -- (0,0);
			    \draw[t1,a1] (1,0) -- (1,1);
			\end{tikzpicture}
		\end{center}
		\caption{The two valid accepting states for building GSAWs on $\GG_2$.}
		\label{figure:height-2-ends}
	\end{minipage}
\end{figure}

Finally, the accepting states are those that contain a single segment (with one or multiple paths) whose end is in the leftmost column. The only two accepting states for the height $2$ case are shown in Figure~\ref{figure:height-2-ends}.

\subsection{A Bijection Between Walks and GSAWs}
\label{subsection:bijection}

Now that we have fully described the construction of the graph $D_h$, we finish this section by proving that we can use $D_h$ to enumerate GSAWs on $\GG_h$.

\begin{theorem}
\label{theorem:bijection}
	Walks in $D_h$ that start at the start vertex, end at an accepting vertex, and have weight $x^ny^k$ (denoted $\WW_{n,k}$) are in bijection with GSAWs with $n$ edges and displacement $k$ (denoted $\SS_{n,k}$). As a consequence, we can find the generating function for GSAWs on $\GG_h$ by applying the transfer matrix method to $D_h$.
\end{theorem}
\begin{proof}
	A walk with displacement $k$ involves $k$ width-two frames. To map a walk $F_1 \to \cdots \to F_k \in \WW_{n,k}$ to a GSAW $S$, we look at $F_1$ to determine which edges with both endpoints having $x$-coordinates $0$ and $1$ are in $S$, and we look at $F_i$ to determine the edges in $S$ that have either both vertices with $x$-coordinate $i$ or one vertex with $x$-coordinate $i-1$ and another with $x$-coordinate $i$. Graphically, this can be thought of as overlapping the right column of each frame with the left column of the next frame. By the way we have defined our start vertex, accepting vertices, and edges and their weights, we are guaranteed that the result is a connected, directed walk with $n$ edges and that its endpoint is trapped. Thus the map $\Phi: \WW_{n,k} \to \SS_{n,k}$ is well-defined. 
	
	To see that $\Phi$ is injective, consider two distinct walks $W_1, W_2 \in \WW_{n,k}$. Suppose that these walks first differ in the $i$th vertex, so we can write
	\begin{align*}
		W_1 &= F_1 \to \cdots \to F_{i-1} \to F_{i} \to \cdots \to F_k,\\
		W_2 &= F_1 \to \cdots \to F_{i-1} \to F'_{i} \to \cdots \to F'_{k}.
	\end{align*}
	If the frames $F_i$ and $F'_i$ have a different underlying set of edges, then $\Phi(W_1)$ and $\Phi(W_2)$ have a different set of edges involving $x$-coordinates $i$ and $i+1$, and thus are different. If they have the same underlying set of edges, then they must be in a different configuration of segments and paths, and this implies those edges must be traversed in a different order in $\Phi(W_1)$ than in $\Phi(W_2)$, and so also in this case are they different.
	
	To prove surjectivity, let $S \in \SS_{n,k}$ be a GSAW. We can break apart $S$ to form the frames $F_1, \ldots, F_{k}$ in the following way. The edges of $F_i$ come from the edges in $S$ whose vertices both have $x$-coordinates $i-1$ and $i$. These edges form some set of paths, which we need to distribute into ordered segments. Whenever two paths are connected by edges in $S$ that all have $x$-coordinate at most $i$, those paths are part of the same segment; otherwise, they are in different segments. The order of the segments and the order of the paths within each segment correspond to the order that their edges are traversed when following $S$ from start to end. Define the walk $W = F_1 \to \cdots \to F_{k}$. By design, each of these edges is an edge in $D_h$, $F_1$ is a neighbor of the designated start vertex, $F_{k}$ is an accepting vertex, and the weight of $W$ is $x^ny^k$ where $n$ is the number of edges in $S$. Thus, $W \in \WW_{n,k}$ and $\Phi(W) = S$.
\end{proof}

\section{Non-Probabilistic Results}
\label{section:non-prob-results}

For each grid graph $\GG_h$, the construction in the previous section provides us with a directed weighted graph $D_h$ and with a designated start vertex and accepting vertices to which we can apply the transfer matrix method to compute the generating function for its valid walks, which according to Theorem~\ref{theorem:bijection} is also the generating function counting GSAWs on $\GG_h$ starting in the top-left corner according to their length and displacement.

\begin{theorem}
	For any $h \geq 1$, the generating function
	\[
		f_h(x,y) = \sum_{W \in W_{(0,h-1)}(\GG_h)} x^{|W|}y^{\|W\|}
	\]
	that counts GSAWs on $\GG_h$ that start in the top-left corner by length and displacement is rational.
\end{theorem}

To carry out the transfer matrix method, one inverts a matrix $M$ whose width and height are equal to the number of vertices in $D_h$.\footnote{It actually suffices to solve a system of the form $Ax=b$, a simpler but still computationally intensive task.} This must be done exactly, not numerically. The number of vertices in $D_h$ grows quickly as we consider taller and taller grid graphs $\GG_h$, but thankfully some simplification is possible. It is well known that given a finite state machine $M_1$, one can compute the unique finite state machine $M_2$ with a minimal number of states that accepts precisely the same set of words. This process is called \emph{minimization}. For the combinatorial finite state machines we have constructed here (i.e., the directed graphs), a similar process can be done, shrinking the dimensions of the matrix $M$ and allowing for faster computation, though the result is not unique. We use a Python package of the second author~\cite{FiniteStateMachinesbibtex} to perform these computations. Table~\ref{table:num-states} on page~\pageref{table:num-states} shows the number of vertices and edges before and after minimization for the finite state machines that we construct. The Github repository~\cite{SelfAvoidingStripWalks-repo} contains all of the Python code used to generate the directed graphs $D_h$ and compute their generating functions.

\subsection{Height 2}

Applying the transfer matrix method to minimization of the graph $D_2$ for GSAWs on the graph $\GG_2$, we find their generating function to be
\begin{align*}
	f_2(x,y) &= \frac{x^3y(1-xy)}{(1-x^2y)(1-xy-x^2y)}\\
	&= x^{3} y+2 x^{5} y^{2}+\left(3 x^{7}+x^{6}\right) y^{3}+\left(4 x^{9}+3 x^{8}+x^{7}\right) y^{4} +\left(5 x^{11}+6 x^{10}+4 x^{9}+x^{8}\right) y^{5} + \cdots\\
	&= y x^{3}+2 y^{2} x^{5}+y^{3} x^{6}+\left(y^{4}+3 y^{3}\right) x^{7}+\left(y^{5}+3 y^{4}\right) x^{8}+\left(y^{6}+4 y^{5}+4 y^{4}\right) x^{9}+ \cdots .
\end{align*}
We can set $y=1$ to find the univariate generating function counting only by length:
\[
	f_2(x, 1) = \frac{x^{3}}{\left(1+x\right) \left(1-x-x^2\right)} = x^3 + 2x^5 + x^6 + 4x^7 + 4x^8 + 9x^9 + 12x^{10}\cdots.
\] 

This tells us that on the grid graph of height $2$, there is a single GSAW of length $3$, two of length $5$, one of length $6$, four of length $7$, and so on. This sequence is (up to a shift) A008346 in the OEIS~\cite{oeis}.

From a generating function, one can obtain a constant-term recurrence for this sequence. Letting $a(n)$ denote the coefficient of $x^n$ in $f_2(x,1)$, we find the formula
\[
	a(n) = 2a(n-2) + a(n-3)
\]
with initial conditions $a(0) = a(1) = a(2) = 0$ and $a(3) = 1$. For any rational generating function, it is also possible to write down an explicit formula for the coefficients but it is expressed as a sum over the roots of the denominator and can be cumbersome, so we do not reproduce it here.

It is also possible to determine the asymptotic growth of the coefficients of any rational generating function (see~\cite[Theorem IV.9]{flajolet:ac}). We calculate that
\[
	a(n) \sim C \cdot \mu^n, \qquad
	C = 1 - \frac{2}{\sqrt{5}} \approx 0.106, \qquad
	\mu = \frac{1+\sqrt{5}}{2} \approx 1.618.
\]
We can also set $x=1$ to find the univariate generating function counting only by displacement:
\[
	f_2(1, y) = \frac{y}{1-2y} = y + 2y^2 + 4y^3 + 8y^4 + 16y^5 + \cdots.
\]
It is easy to see independently that there are $2^{k-1}$ GSAWs with displacement $k$ on the height $2$ half-infinite grid graph. Each walk $W$ with displacement $k-1$ can be extended to two walks $W_1$ and $W_2$ with displacement $k$: $W_1$ is formed by shifting $W$ one unit to the right, prepending the step $(0,1) \to (1,1)$ and if $W$ ends at $(0,0)$ then appending the step $(1,0) \to (0,0)$ to the end of $W_1$; $W_2$ is formed by shifting $W$ one unit to the right, flipping it vertically, the prepending the steps $(0,1) \to (0,0) \to (1,0)$. The two extensions of each walk are different and distinct from the extensions of each other walk, and so there are twice as many walks of displacement $k$ as of displacement $k-1$.

\subsection{Height 3}

The generating function for GSAWs on the graph $\GG_3$ is
\begin{align*}
	f_3(x,y) &= \frac{x^4y \cdot p(x,y)}{q_1(x,y) \cdot q_2(x,y)}
\end{align*}
where
\begin{align*}
	p(x,y) = 1&+x -x y -3 x^{2} y + 2 x^{3} y^{2} + 2 x^{4} y -2 x^{4} y^{2} -x^{5} y^{2} + 2 x^{6} y^{3} -x^{7} y^{2} + x^{7} y^{3} -x^{7} y^{4} -3 x^{8} y^{3}\\
		& + x^{8} y^{4} -3 x^{9} y^{4} -x^{10} y^{4} + x^{10} y^{5} + 4 x^{11} y^{4} -x^{11} y^{5} +2 x^{13} y^{5},\\
	q_1(x,y) &= 1 -2 x y + x^{2} y^{2} -2 x^{3} y + x^{4} y^{2} -x^{5} y^{3} + 2 x^{7} y^{3},\\
	q_2(x,y)&= 1- 2 x^{2} y + x^{4} y^{2} - 2 x^{6} y^{2}.
\end{align*}

We can set $y=1$ to find the univariate generating function counting only by length:
\begin{align*}
	f_3(x, 1) &= \frac{x^{4}\left(1+x -2 x^{2}-x^{5}+x^{6}-2 x^{8}-5 x^{9}-5 x^{10}-2 x^{11}-2 x^{12}\right) }{\left(1-x -2 x^{3}-x^{4}-2 x^{5}-2 x^{6}\right) \left(1-2 x^{2}\right) \left(1+x^{4}\right)}\\[5pt]
	&= x^4 + 2x^5 + 2x^6 + 6x^7 + 10x^8 + \cdots.
\end{align*}

We have added this sequence as A374297 in the OEIS. Letting $a(n)$ denote the number of such GSAWs of length $n$, we find the recurrence
\begin{align*}
	a(n) &= a(n-1) + 2a(n-2) - a(n-5)+2a(n-6)-4a(n-7)\\
	& \qquad -3a(n-8)-2a(n-9)-4a(n-11)-4a(n-12)
\end{align*}
with the appropriate initial conditions.

Let $\alpha$ denote the root of $1-x -2 x^{3}-x^{4}-2 x^{5}-2 x^{6}$ near $0.522$. Then,
\[
	a(n) \sim C \cdot \mu^n,
\]
where $C \approx 0.061$ is a rational function of $\alpha$ and $\mu = 1/\alpha \approx 1.915$.

We can also set $x=1$ to find the univariate generating function counting only by displacement:
\[
	f_3(1, y) =\frac{2 y \left(1+y\right)(1-y-y^3)}{(1-3y-y^2)(1-2y-y^2)} = 2 y +10 y^{2}+40 y^{3}+148 y^{4}+526 y^{5} + \cdots.
\]

We have added this sequence as A374298 in the OEIS.

\subsection{Height 4}

The generating function for GSAWs on the graph $\GG_4$ is
\[
	f_4(x,y) = \frac{x^5y \cdot p(x,y)}{q_1(x,y) \cdot q_2(x,y)}
\]
where $p$, $	q_1$, and $q_2$ have total degrees $51$, $25$, and $26$ respectively. The polynomials are too large to print here, but they can be found in an ancillary file attached to the arXiv version of this article and in the Github repository~\cite{SelfAvoidingStripWalks-repo}.

We can set $y=1$ to find the univariate generating function counting only by length:
\[
	f_4(x,1) = \frac{x^5 \cdot r(x)}{s_1(x) \cdot s_2(x)}
\]
where
\begin{align*}
	r(x) &= 3-4 x -7 x^{2}-3 x^{3}+21 x^{4}+19 x^{5}-10 x^{6}-45 x^{7}-22 x^{8}+17 x^{9}+55 x^{10}+17 x^{11}\\
	& \qquad -28 x^{12}+21 x^{13}-5 x^{14}+74 x^{15}-143 x^{16}-57 x^{17}-98 x^{18}+150 x^{19}+208 x^{20}\\
	& \qquad -72 x^{21}-91 x^{22}-246 x^{23}+144 x^{24}+181 x^{25}+66 x^{26}+11 x^{27}-242 x^{28}\\
	& \qquad +38 x^{29}+49 x^{30}-3 x^{31}+38 x^{32}-107 x^{33}+12 x^{34}+45 x^{35}+18 x^{36}+20 x^{37}\\
	& \qquad -14 x^{38}-12 x^{39},\\
	s_1(x) &= 1-2 x -3 x^{3}+6 x^{4}-x^{5}+2 x^{6}-8 x^{7}+3 x^{8}+10 x^{9}+3 x^{10}-8 x^{11}-17 x^{12}+7 x^{13},\\
	&\qquad +8 x^{14}+5 x^{15}-6 x^{16}-7 x^{17}+2 x^{18}+2 x^{19}\\
	s_2(x) &= 1-4 x^{2}+3 x^{4}+2 x^{6}-4 x^{8}+2 x^{10}-x^{12}+8 x^{14}-5 x^{16}-2 x^{18}+4 x^{20}.
\end{align*}

The series expansion starts
\[
	f_4(x,1) = 3 x^{5}+2 x^{6}+9 x^{7}+8 x^{8}+36 x^{9}+ \cdots,
\]
which we have added as sequence A374299 in the OEIS. The recurrence for $a(n)$ involves the terms from $a(n-1)$ to $a(n-39)$, so we will not reproduce it here.

Let $\alpha$ denote the root of $s_1(x)$ near $0.479$. Then,
\[
	a(n) \sim C \cdot \mu^n,
\]
where $C \approx 0.042$ is a rational function of $\alpha$ and $\mu = 1/\alpha \approx 2.087$.

We can also set $x=1$ to find the univariate generating function counting only by displacement:
{\footnotesize
\begin{align*}
	f_4(1, y) &= \frac{y(5-31 y+85 y^{2}-132 y^{3}+112 y^{4}+93 y^{5}-485 y^{6}+552 y^{7}-119 y^{8}-205 y^{9}+138 y^{10}-4 y^{11}-11 y^{12})}{(1-7 y+12 y^{2}-7 y^{3}+3 y^{4}+2 y^{5})(1-8 y+15 y^{2}-5 y^{3}-9 y^{4}+2 y^{5}+y^{6})}\\
	&= 5 y +44 y^{2}+330 y^{3}+2231 y^{4}+14234 y^{5}+87670 y^{6}+526549 y^{7}+3105097 y^{8}+18061476 y^{9}+ \cdots.
\end{align*}
}

We have added this sequence as A374300 in the OEIS.

\subsection{Height 5}

The directed graph that we create to enumerate GSAWs on the graph $\GG_5$ has 7,201 vertices and 78,408 edges before minimization and 152 vertices and 2,142 edges after minimization. We find the generating function to be
\[
	f_5(x,y) = \frac{x^5y \cdot p(x,y)}{q_1(x,y) \cdot q_2(x,y)}
\]
where $p$, $	q_1$, and $q_2$ have total degrees $151$, $66$, and $84$ respectively. The polynomials are too large to print here, but they can be found in an ancillary file attached to the arXiv version of this article and in the Github repository~\cite{SelfAvoidingStripWalks-repo}.

We can set $y=1$ to find the univariate generating function counting only by length:
\[
	f_5(x,1) = \frac{x^5 \cdot r(x)}{s_1(x) \cdot s_2(x)}
\]
where $r$, $s_1$, and $s_2$ are irreducible polynomials with degrees $123$, $68$, and $54$, respectively. The series expansion begins
\[
	f_5(x,1) = 2x^5 + 3x^6 + 8x^7 + 13x^8 + 32x^9 + \cdots.
\]
We have added this as sequence A374301 in the OEIS. The recurrence for $a(n)$ involves the terms $a(n-1)$ to $a(n-122)$.

Let $\alpha$ denote the root of the denominator of $f_5(x,1)$ near $0.454$. Then,
\[
	a(n) \sim C \cdot \mu^n,
\]
where $C \approx 0.033$ is a rational function of $\alpha$ and $\mu = 1/\alpha \approx 2.199$.

We can also set $x=1$ to find the univariate generating function counting only by displacement:
\[
	f_5(1, y) = \frac{y \cdot u(y)}{v_1(y) \cdot v_2(y)}
\]
where
\begin{align*}
	u(y) = 11&-169 y+1221 y^{2}-5279 y^{3}+13594 y^{4}-11881 y^{5}-57814 y^{6}+247985 y^{7}-333951 y^{8}\\
			&-181528 y^{9}+986287 y^{10}-622581 y^{11}-918040 y^{12}+1369959 y^{13}-170227 y^{14}\\
			&-734832 y^{15}+861370 y^{16}-531077 y^{17}-591839 y^{18}+825566 y^{19}+302374 y^{20}\\
			&-441374 y^{21}-167471 y^{22}+107296 y^{23}+76977 y^{24}+16189 y^{25}+970 y^{26}\\
	v_1(y) = \phantom{1}1&-13 y+45 y^{2}-66 y^{3}+17 y^{4}+209 y^{5}-151 y^{6}-140 y^{7}+112 y^{8}+48 y^{9}-50 y^{10}-28 y^{11}\\
	v_2(y) = \phantom{1}1&-18 y+105 y^{2}-237 y^{3}+27 y^{4}+650 y^{5}-500 y^{6}-362 y^{7}+260 y^{8}+96 y^{9}+118 y^{10}\\
		&-130 y^{11}-74 y^{12}-23 y^{13}-3 y^{14}.
\end{align*}
The series expansion begins
\[
	f_5(1,y) = 11 y +172 y^{2}+2329 y^{3}+28130 y^{4}+318086 y^{5}+3454914 y^{6}+36484161 y^{7}+\cdots.
\]
We have added this as sequence A374302 in the OEIS.

\subsection{Height 6}

The directed graph that we create to enumerate GSAWs on the graph $\GG_6$ has 80,378 vertices and 1,518,388 edges before minimization and 573 vertices and 17,144 edges after minimization. We find the generating function to be
\[
	f_6(x,y) = \frac{x^5y \cdot p(x,y)}{q_1(x,y) \cdot q_2(x,y)}
\]
where the total degrees of $p$, $q_1$, and $q_2$ are $443$, $199$, and $240$ respectively. 

We can set $y=1$ to find the univariate generating function counting only by length:
\[
	f_6(x,1) = \frac{x^5 \cdot r(x)}{s_1(x) \cdot s_2(x)}
\]
where $r$, $s_1$, and $s_2$ are irreducible polynomials with degrees $374$, $168$, and $202$, respectively. The series expansion begins
\[
	f_6(x,1) = 2x^5 + 2x^6 + 9x^7 + 10x^8 + 40x^9 + \cdots.
\]
and can be found in the OEIS entry A374303. The recurrence for $a(n)$ involves the terms $a(n-1)$ to $a(n-370)$.

Let $\alpha$ denote the root of the denominator of $f_6(x,1)$ near $0.439$. Then,
\[
	a(n) \sim C \cdot \mu^n,
\]
where $C \approx 0.027$ is a rational function of $\alpha$ and $\mu = 1/\alpha \approx 2.276$.

We can also set $x=1$ to find the univariate generating function counting only by displacement:
\[
	f_6(1, y) = \frac{y \cdot u(y)}{v_1(y) \cdot v_2(y)}
\]
where $u$, $v_1$, and $v_2$ have degrees $69$, $31$, and $38$ respectively, and the series expansion starts
\[
	f_6(1,y) = 23 y +629 y^{2}+15134 y^{3}+323031 y^{4}+6428665 y^{5}+122523673 y^{6}+2267420832 y^{7}+ \cdots 
\]
This is sequence A374304 in the OEIS.

\subsection{Height 7}

For height $7$, the directed graph $D_7$ has $954,791$ vertices before minimizing and $2311$ vertices after. We are not able to find the full solution $f_7(x,y)$ but we can find the specializations $f_7(x,1)$ and $f_7(1,y)$.\footnote{Although Maple's built-in linear algebra methods were not able to perform the necessary calculations, a custom script we wrote succeeded after three days.} We can write
\[
	f_7(x,1) = \frac{x^5 r(x)}{s_1(x)s_2(x)}
\]
where $r(x)$, $s_1(x)$, and $s_2(x)$ have degrees $1071$, $436$, and $628$. The series expansion begins
\[
	f_7(x,1) = 2 x^{5}+2 x^{6}+8 x^{7}+11 x^{8}+34 x^{9}+70 x^{10}+\cdots
\]
and we have added this as sequence A374305 in the OEIS. Letting $\alpha$ denote the root of the denominator near $0.429$, we have
\[
	a(n) \sim C \cdot \mu^n
\]
where $C \approx 0.024$ is a rational function of $\alpha$ and $\mu = 1/\alpha \approx 2.332$.

We can also write
\[
	f_7(1,y) = \frac{y \cdot u(y)}{v_1(y) \cdot v_2(y)}
\]
where $u$, $v_1$, and $v_2$ have degrees $162$, $65$, and $96$ respectively. The series expansion starts
\[
	f_7(1,y) = 47 y +2221 y^{2}+94006 y^{3}+3527224 y^{4}+123159829 y^{5}+4110628551 y^{6}+\cdots,
\]
and this is sequence A374306 in the OEIS.

We note that for heights 2 through 6, the connective constants $\mu$ that we computed are the same as those found by Klein \cite{klein:saws} in an enumeration of confined self-avoiding walks.

\section{The Probabilistic Finite State Machine Constructions}
\label{section:fsm-prob}

Klotz and Sullivan~\cite{klotz:gsaws1} studied GSAWs on the height $2$ grid graph from a probabilistic viewpoint. Each time a GSAW takes a step, it chooses from among its neighbors according to a probability distribution. Then, the probability of the entire GSAW is the product of the probabilities of each step. In this way we can measure how some GSAWs are more likely to be the outcome of this process than others.

In the first of the probability distributions considered by Klotz and Sullivan, the next vertex visited by a GSAW is chosen uniformly among all open neighbors. For example, the GSAW shown in Figure~\ref{figure:first-prob-example} on page~\pageref{figure:first-prob-example} has probability $1/576$ of occurring among all GSAWs on $\GG_3$.

\begin{figure}
	\begin{center}
		\begin{tikzpicture}[
			scale=0.75,
			baseline=(current bounding box.center),
		]
			\draw[gray] (0,0) grid (5.5,4);
			\foreach \x in {0,1,2,3,4,5} { \foreach \y in {0,1,2,3,4} {
				\draw[gray, fill=white] (\x,\y) circle (1.9pt);
			}}
			
			\def\points{(0,4),(1,4),(2,4),(3,4),(4,4),(4,3),(4,2),(3,2),(3,3),(2,3),(1,3),(0,3),(0,2),(0,1),(0,0),(1,0),(2,0),(2,1)}

			\draw[ultra thick] (0,4)
				\foreach \point in \points {
					-- \point
				};
			\foreach \point in \points {
				\draw[black, fill] \point circle (2pt);
			}
			
			\node[above right] at (2,1) {$v$};
  		\end{tikzpicture}
	\end{center}
	\caption{A partially built GSAW on $\GG_5$.}
	\label{figure:prob-ex-2}
\end{figure}
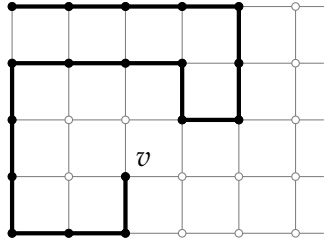

The second probability distribution assigns each neighbor an \emph{energy} depending on how many of its neighbors are vertices already in the GSAW, and then chooses among the neighbors with a probability corresponding to the energies. A preference for situating near already-occupied sites mimics the behavior of polymers in poor-solvent conditions, in which non-solvation between the polymer and its solvent is resolved by having the polymer condense into a tight globule. Recently, Hooper and Klotz~\cite{hooper:trapping-saws} explored the effect of neighbor attraction of GSAW statistics, finding a global minimum in the mean trapping length as a function of the attraction strength. 

Let $C$ be a positive constant, and let $N(v)$ denote the unvisited neighbors of a vertex $v$. If $v$ is the endpoint of a partially built GSAW $G$ and $w \in N(v)$, then the energy of $w$, denoted $E(w)$, is $C^\ell$, where $\ell$ is the number of neighbors of $w$ that are part of $G$, not counting $v$. For example, consider the vertex $v = (2,1)$ in Figure~\ref{figure:prob-ex-2}. The unvisited neighbors of $v$ are $N(v) = \{(1,1), (2,2), (3,1)\}$, and the energies of these unvisited neighbors are $C^2$, $C^2$, and $C$, respectively. The probability that a vertex $w \in N(v)$ is the next vertex visited by the GSAW is the ratio of its energy to the total energy of all unvisited neighbors of $v$:
\[
	\frac{E(w)}{\ds\sum_{w' \in N(v)}E(w')}.
\]
For example the probability that $(1,1)$ is the next vertex visited in the GSAW shown in Figure~\ref{figure:prob-ex-2} is $C^2/(2C^2+C) = C/(2C+1)$. The total probability of a GSAW is defined to be the product of the probabilities of each step.  While greater positive values of $C$ can model attraction between segments of a polymer due to solvent effects, negative values may represent repulsive forces stronger than simple hard-sphere repulsion, such as interaction between like-charged monomers.

We call these two probability models the \emph{uniform model} and the \emph{energetic model}. Note that the energetic model is a generalization of the uniform model, which can be recovered from the energetic model by setting $C=1$. Our finite state machine construction can be adapted to permit the computation of the generating functions for GSAWs on half-infinite grid graphs of bounded height under both of these models.

\subsection{The Uniform Model}

In Subsection~\ref{subsection:non-prob-construction} we described how a width $2$ frame $F$ could be extended to a set $E'(F)$ of extended width $3$ frames, which were then trimmed down to form the set $N(F)$ of width $2$ frames that are the neighbors of $F$ in the directed graph being built. For a frame $F$ and its neighbor $F' \in N(F)$ we previously assigned the edge $F \to F'$ a weight $x^ky$ where $k$ is the number of edges added while extending $F \to F'$. Now we must also incorporate probabilities into the edge weights. 

At first this seems like an impossible task because for edges $v_1 \to v_2$ in the rightmost column in $F'$, we do not yet know whether the neighbor to the right of $v_1$ is occupied. We have not yet built the frame containing that neighbor, but when we do, it might contain edges of the GSAW that come before the edge $v_1 \to v_2$. For example, the probability of the edge $e = (1,1) \to (1,2)$ in Figure~\ref{figure:uniform-ex-1} depends on whether the vertex to the right of $(1,1)$ is occupied by an edge that comes earlier in the GSAW than $e$. 

\begin{figure}
	\begin{center}
		\begin{tikzpicture}[
			scale=0.75,
			baseline=(current bounding box.center),
		]
			\draw[lightgray] (-0.5,-0.5) grid (1.5,4.5);
			\graydots{}
			\draw[ultra thick] (-0.5,-0.5) rectangle (1.5,4.5);
			
		    \draw[t1, a1, circ] (0,4) -- (1,4); 
		    \draw[t2, a1] (0,2) -- (0,1);
			\draw[t2, sq] (1,3) -- (0,3) -- (0,2) -- (0,1) -- (0,0) -- (1,0) -- (1,2);
		\end{tikzpicture}

	\end{center}
	\caption{An example of a frame. The probability of the edge $(1,1) \to (1,2)$ cannot be computed until this frame is extended.}
	\label{figure:uniform-ex-1}
\end{figure}
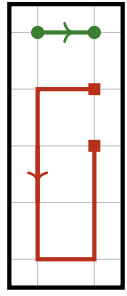

To remedy this, we instead account, in each transition $F \to F'$, for the probabilities of the edges whose start vertex is in the right column of $F$ (which is the left column of $F'$). To do this, when we have extended $F$ to some width $3$ frame in $E'(F)$, but before trimming it, we can see all neighbors of vertices in the middle column, and therefore compute the probabilities of all edges that originate in the middle column (regardless of whether they terminate in the left, middle, or right columns). The probabilities of these edges must be calculated in the order in which they appear in the GSAW; in other words, we must consider the segments in order, and the paths of each segment in order, and the edges of each path in order. We thus calculate the product of the probabilities of all such edges, and multiply this by the existing edge weight accounting for the number of new edges added in the transition. Figure~\ref{figure:uniform-ex-2} shows an example of this procedure. When $F'$ is a frame corresponding to an accepting vertex (i.e., we have placed all edges of the GSAW), we also incorporate the probabilities of edges originating in the right column.

\begin{figure}
	\begin{center}
		\begin{tikzpicture}[
			scale=0.75,
			baseline=(current bounding box.center),
		]
			\draw[lightgray] (-0.5,-0.5) grid (2.5,4.5);
			\draw[ultra thick] (-0.5,-0.5) rectangle (2.5,4.5);
			
		    \draw[t1, circ] (0,4) -- (2,4);
		    \draw[t1, a1] (1.09, 4) -- (1.1,4);
		    \draw[t2, sq] (2,0) -- (0,0);
		    \draw[t2, a1] (0.91,0) -- (0.9,0);
		    \draw[t2, sq] (0,1) -- (1,1) -- (1,2) -- (0,2);
		    \draw[t2, a2] (1,1) -- (1,2);
		    \draw[t2, sq] (0,3) -- (2,3);
		    \draw[t2, a3] (1.25,3) -- (1.26,3);
		    \draw[t3, diam, a1] (2,1) -- (2,2);		      
		\end{tikzpicture}
		\quad
	\end{center}
	\caption{The edges that originate in the middle column are, in the order they appear in the GSAW, $(1,4) \to (2,4)$, $(1,0) \to (0,0)$, $(1,1) \to (1,2)$, $(1,2) \to (0,2)$, and $(1,3) \to (2,3)$. The probability of $(1,4) \to (2,4)$ is $1/2$ because both neighbors $(1,3)$ and $(2,4)$ are unoccupied at this point. The probability of $(1,0) \to (0,0)$ is similarly $1/2$. The probability of $(1,1) \to (1,2)$ is $1/2$ because the unoccupied neighbors of $(1,1)$ are $(2,1)$ and $(1,2)$, as $(1,0)$ is now occupied. The probability of $(1,2) \to (0,2)$ is $1/3$ because $(1,2)$ has three unoccupied neighbors. Finally, the probability of $(1,3) \to (2,3)$ is $1$ because at this point $(2,3)$ is the only unvisited neighbor of $(1,3)$.}
	\label{figure:uniform-ex-2}
\end{figure}
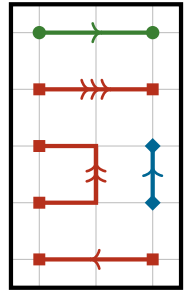

As a result, for each grid graph $\GG_h$ we can form a graph $D_h^u$ with the same vertex set, edge set, start vertex, and accepting vertices as $D_h$, but whose edge weights are each multiplied by a rational number corresponding to the probability of that transition. In Section~\ref{section:prob-results} we will apply the transfer matrix method to these graphs to produce generating functions that count GSAWs according to the uniform model.

\begin{theorem}
	Let $h \geq 1$ and let $p_h^U(W)$ be the probability that a GSAW $W$ on $\GG_h$ occurs under the uniform model. Then, the generating function
	\[
		f^U_h(x,y) = \sum_{W \in W_{(0,h-1)}(\GG_h)} p^U_h(W)x^{|W|}y^{\|W\|}
	\]
	that counts GSAWs on $\GG_h$ that start in the top-left corner by length and displacement is rational.
\end{theorem}

\subsection{The Energetic Model}

In the energetic model, the probability of moving to a new vertex depends not just on whether the neighbors of the new vertex are already in the GSAW, but whether the neighbors of the neighbors are in the GSAW. For this reason, the probabilities of each edge cannot be computed using width $2$ frames, which during the extension process only have $3$ columns. Instead we must work entirely with width $4$ frames so that during extension we see $5$ columns. During the extension process we can compute the energetic probabilities of all edges originating in the middle column because we have within the extended frame all neighbors of neighbors.

Working with width $4$ frames is a significant computational burden because there are many more width $4$ frames than width $2$ frames. (See Table~\ref{table:num-states} on page~\pageref{table:num-states} to compare the number of states required for each probabilistic model.) Despite this, the process of computing extensions is unchanged from the width $2$ case, and so we are able to obtain graphs $D_h^e$ for each grid graph $\GG_h$ (although with many more vertices than $D_h^u$) and apply the transfer matrix method to compute the corresponding generating functions. 

\begin{theorem}
	Let $h \geq 1$ and let $p_h^E(W)$ be the probability that a GSAW $W$ on $\GG_h$ occurs under the energetic model. Then, the generating function
	\[
		f^E_h(x,y,C) = \sum_{W \in W_{(0,h-1)}(\GG_h)} p^E_h(W)x^{|W|}y^{\|W\|}
	\]
	that counts GSAWs on $\GG_h$ that start in the top-left corner by length and displacement is rational.
\end{theorem}

\section{Probabilistic Results}
\label{section:prob-results}

\subsection{Height 2}
Applying the transfer matrix method to the graph with modified edge weights, we find the generating function for the uniform model on $\GG_2$ to be\footnote{Klotz and Sullivan~\cite{klotz:gsaws1} previously computed all of the height 2 results presented in this subsection.
}
\begin{align*}
	f^U_2(x,y) &= \frac{x^{3} y \left(2-x y\right)}{\left(2-x^{2} y\right)\left(8-4xy-4x^2y+x^3y^2\right)}\\
	&= \frac{1}{8}x^{3}y
		+\frac{1}{8} x^{5}y^{2} 
		+\frac{1}{64}  x^{6} y^{3}
		+\frac{3}{32} x^7y^{3}
		+\frac{1}{128} x^7y^{4}
		+ \frac{3}{128} x^8y^{4}
		+\frac{1}{256} x^8y^{5} + \cdots,
\end{align*}
which reveals that the probability that a walk becomes trapped with length $3$ and displacement $1$ is $\frac{1}{8}$ (in this case there is a single such walk), the probability a GSAW becomes trapped with length $5$ and displacement $2$ is also $\frac{1}{8}$ (in this case there are two walks, each with probability $\frac{1}{16}$), and so on. Note that $f_2(1,1) = 1$, demonstrating that a GSAW in this model has finite length and displacement with probability $1$.

As before, one can substitute either $x=1$ or $y=1$ to find univariate rational generating functions for GSAWs in this model by just length or just displacement, and then can turn these into recurrences, asymptotics, etc., but we will spare the reader these results. However, we can also use these generating functions to determine the expected value of the length or displacement of a GSAW, or in fact any moment of these two quantities. The expected length of a GSAW on $\GG_2$ is
\[
	\left(\frac{d}{dx}\;f^U_2(x, 1)\right)_{x=1} = 13
\]
and the expected displacement is
\[
	\left(\frac{d}{dy}\;f^U_2(1, y)\right)_{y=1} = 7.
\]
The corresponding variances are
\[
	\left(\frac{d^2}{dx^2}\;f^U_2(x, 1) + \frac{d}{dx}\;f^U_2(x, 1) - \left(\frac{d}{dx}\;f^U_2(x, 1)\right)^2\right)_{x=1} = 98
\]
and
\[
	\left(\frac{d^2}{dy^2}\;f^U_2(1, y) + \frac{d}{dy}\;f^U_2(1, y) - \left(\frac{d}{dy}\;f^U_2(1, y)\right)^2\right)_{y=1} = 40.
\]

The solution to the energetic model is harder to find because it requires the use of width-four frames instead of width-two frames, leading to many more vertices in the directed graph and thus a much larger matrix equation that must be solved. Whereas the directed graph for the uniform model on $\GG_2$ has only 12 vertices before minimization and 7 vertices after, the directed graph for the energetic model has 54 vertices before minimization and 24 vertices after. Even worse, the weights of the edges involve an extra variable $C$ which makes the symbolic computations more intensive.

We find the generating function for the energetic model to be
\[
	f^E_2(x,y,C) = \frac{Cx^3y \cdot p(x,y,C) }{2(C+1)(2-x^2y) \cdot q(x,y,C)}
\]
where
\begin{align*}
	p(x, y, C) &=
		4 \left(C^{2}+1\right) \left(C +1\right)^{2}
		-2 \left(C^{2}+1\right) \left(C +1\right)^{2} xy
		-4 C^{2} \left(C -1\right) \left(C +1\right) x^{2}y\\
		&\qquad
		+2 \left(C -1\right) \left(C +1\right) \left(C^{2}+1\right) x^{3}y^2
		-2 \left(C -1\right) \left(C +1\right) x^{4} y^{2}\\
		& \qquad
		-C \left(C -1\right) x^{4} y^{3}
		+C \left(C -1\right) x^{5}y^3\\
	q(x,y, C) &= 
		4 \left(C^{2}+1\right) \left(C +1\right)^{2}
		-2 \left(C^{2}+1\right) \left(C +1\right)^{2} xy
		-4 C \left(C +1\right) \left(C^{2}+1\right) x^{2} y\\
		&\qquad 
		+2 \left(C +1\right) \left(C^{3}+C -1\right) x^{3} y^{2}
		-C \left(C -1\right) x^{4} y^{3}.
\end{align*}

The series expansion begins
\[
	f_2^E(x,y, C) = \frac{C}{4(C+1)}x^3y + \frac{C \left(C +1\right)}{8 C^{2}+8}x^5y^2 + \frac{C^{2}}{8 \left(C^{2}+1\right) \left(C +1\right)^{2}}x^6y^3 + \cdots.
\]
The expected length is
\[
	\frac{4 C^{5}+36 C^{4}+59 C^{3}+51 C^{2}+42 C +16}{2 C \left(2 C^{2}+C +1\right) \left(C +1\right)},
\]
the expected displacement is
\[
	\frac{C^{4}+7 C^{3}+8 C^{2}+6 C +6}{C \left(2 C^{2}+C +1\right)},
\]
and the variances of length and displacement are
{ \small
\[
	\frac{48 C^{10}+288 C^{9}+1000 C^{8}+2292 C^{7}+3769 C^{6}+4864 C^{5}+4981 C^{4}+3998 C^{3}+2488 C^{2}+1104 C +256}{4 C^{2} \left(2 C^{2}+C +1\right)^{2} \left(C +1\right)^{2}}
\]
}
and
\[
	\frac{\left(C +1\right) \left(3 C^{7}+9 C^{6}+33 C^{5}+59 C^{4}+62 C^{3}+84 C^{2}+34 C +36\right)}{C^{2} \left(2 C^{2}+C +1\right)^{2}}
\]
respectively.

\subsection{Heights 3, 4, 5, and 6}

We can perform the same calculations for the uniform model for heights up to 6 and for the energetic model for heights up to 4. In this subsection we briefly collect the results that fit within these margins. As in the non-probabilistic case, full generating functions can be found in an ancillary file attached to the arXiv version of this article.

For height $3$, the directed graph for the uniform model has 79 vertices before minimization and 30 vertices after, while the directed graph for the energetic model has 1078 vertices before minimization and 225 vertices after. The total degrees of the numerator and denominator of the generating function $f^U_3(x,y)$ for the uniform model are 41 and 43. The expected length is $486414801/25317287 \allowbreak \approx 19.21$, the variance of length is 
	$120539533380801646/640965021040369 \allowbreak \approx 188.06$, the expected displacement is $3013/389 \allowbreak \approx 7.75$ and the variance of displacement is $13442841/302642 \allowbreak \approx 44.42$.
	
Although it takes only a few minutes to construct the directed graph for the energetic model, the linear algebra operations to compute the generating function take several hours. The total degrees of the numerator and denominator (in $x$, $y$, and $C$) are $107$ and $102$. The expectations and variances of length and displacement are rather large rational functions of $C$. For example, the expected length is
\[
	\frac{p(C)}{16 \left(2 C +1\right) \left(C +2\right) \left(C +1\right) \left(C^{2}+1\right)q(C)}
\]
where $p(C)$ is a polynomial of degree $51$ with coefficients up to $19$ digits long and $q(C)$ is a polynomial of degree $45$ with coefficients up to $15$ digits long.

For height $4$ the minimized directed graph for the uniform model has $162$ vertices. The total degrees of the numerator and the denominator of the generating function $f^U_4(x,y)$ are $242$ and $244$. The expected length is 
\[
	\frac{36398762646457399797278201924184927631185786219507599165631529}{1584271961568506676402164217848452838298247680057676502545715} \approx 22.98.
\]
and the expected displacement is $163501511557499670817/20874575314535639327 \approx 7.83$. The variances of length and displacement are approximately $253.71$ and $44.00$ respectively. 


The minimized directed graph for the energetic model has $3387$ vertices, and we are unable to perform the linear algebra operations to obtain $f^E_4(x,y,C)$. We can, however find the specializations $f^E_4(x,1,C)$ and $f^E_4(1,y,C)$. These generating functions are too long to print here, but they can be found in an ancillary file attached to the arXiv version of this article and in the Github repository~\cite{SelfAvoidingStripWalks-repo}.

For height $5$ the minimized directed graph for the uniform model has $1104$ vertices. The full solution $f^U_5(x,y)$ has numerator and denominator with total degrees 1250 and 1252. The specialization function $f_5^U(x,1)$ has numerator and denominator of degrees $1002$ and $1003$ respectively, and we can use it to find that the expected length of a GSAW in this model is $a/b \approx 26.518$ where $a$ and $b$ have $366$ and $364$ digits and the variance of length is $c/d \approx 334.912$ where $c$ and $d$ have $731$ and $728$ digits. The specialization $f^U(1,y)$ has numerator and denominator of degrees $68$ and $69$ respectively, and we can use it to find that the expected displacement of a GSAW in this model is $e/f \approx 8.084$ where $e$ and $f$ have $102$ and $101$ digits and the variance of length is $g/h \approx 45.533$ where $g$ and $h$ have $204$ and $202$ digits.

For heights $6$ and $7$ the minimized directed graphs for the uniform model have  $7294$ and $53,808$ vertices. We are not able to calculate either of the specializations of $f^U_6$ or $f^U_7$. We can provide estimates of the growth rate of the size of the directed graphs as a function of the strip height. For the non-probabilistic model, the number of vertices grows as approximately $3.6^h$ and the number of edges as $7.3^h$, after minimization. For the uniform model, the vertices and edges grow at approximately $6.5^h$ and $11.3^h$. 

In the context of the energetic model, the above results for the mean trapping length apply in the specific case of $C=1$. In the large-$C$ limit, the walks are highly attractive and the trapping length diverges as cavities are unlikely to form. In the limit that $C$ approaches 0, Laforge et al. \cite{laforge:trapping} recently showed numerically that the trapping length on an infinite square lattice will reach a plateau. By simplifying our generating functions with $C=0$ we can derive the trapping plateau at heights 3 through 6, as the plateau does not exist strips of height 2.  We find it is $3867/112\approx34.5$ for $h=3$, and approximately $33.8$ for $h=4$ and $38.9$ for $h=5$, compared to the unconfined square lattice value of approximately $178.5$, or $107.9\pm0.3$ in the quarter-infinite plane. The non-monotonic trapping length with respect to height likely arises from parity effects, and is also observed in Monte Carlo simulations. On the infinite honeycomb lattice, Laforge et al. observed a trapping plateau for attractive walks as well as repulsive, which may be observed in an extension of our work to the honeycomb (or topologically equivalent brickwork) lattice.

\subsection{Extension to Unconfined Walks}
Our goal is to use results from exactly solvable restricted GSAWs to provide insights about trapping in the infinite-grid GSAW. Since our results pertain to walks beginning in the corner of a semi-infinite strip, they are most comparable to the GSAW beginning in the corner of a quarter-infinite square lattice. Results for trapping in such a lattice are not as well-established as those in an infinite square lattice \cite{hemmer:saw-average-71}, so we perform Monte Carlo simulations to determine what the asymptotic limits would be, were we able to compute the expected trapping length as the strip width tends to infinity. Details of and results from the Monte Carlo simulations are found in the Appendix. The pertinent result is that in the quarter-infinite square lattice, the mean trapping length is $N_\infty= 45.3997248 \pm 0.0003198$.

With exact values of the trapping length between $h=2$ and $h=5$, it is difficult to predict exactly what the unconfined trapping length would be. We can surmise that the trapping length would continue to increase with strip height, but would asymptotically approach its unconfined value. This is supported by the Monte Carlo data, which shows the trapping length converging (within statistical uncertainty) to its asymptotic value around $h=50$, although we wish to base conclusions only on exact results. A simple function that would describe such a trend is an exponential decay of the form:
\[
N(h)=N_\infty \left(1-e^{-\kappa h}\right),
\]
where $N_\infty$ and $\kappa$ may be treated as fit parameters. By fitting this function to the mean trapping lengths from 2 to 5, we find that $N_\infty=45.81$ and $\kappa=0.17$. Although two parameters fit to four data points is not a statistically strong measurement, it is encouraging that our predicted $N_\infty$ is so close to the Monte Carlo measurement. We stress that the exponential decay was an ansatz and may not describe the true functional form of $N(h)$. We hope that as exact values are computed for larger strips, we can better understand the approach towards the unconfined case.

\section{Enumerating Maximal Length GSAWs}
\label{section:full-GSAWs}

In April 1999, a user with the name Thomas Womack posted a question on the sci.math newsgroup about what they called \emph{Greek key tours}, which they defined as ``path[s] visiting all squares of an orthogonally-connected chessboard''~\cite{womack:greek-key-tours}. In our terminology, a Greek key tour is a GSAW on a finite $m \times n$ grid graph that starts in the upper-left corner and has the maximum possible length, i.e., it visits every vertex once. The name ``Greek key tour'' presumably comes from the design pattern of the same name, shown for example in Figure~\ref{figure:greek-key}. This post computed the number of Greek key tours on the $3 \times n$ grid for $n$ up to $12$ and conjectured the generating function for the sequence. There are several posts in reply, but it appears that the conjecture was not proven at the time. 

\begin{figure}
	\begin{center}
		\includegraphics[width=0.5\linewidth]{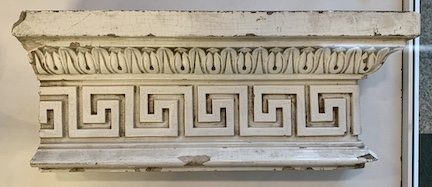}
	\end{center}
	\caption{A stove tile from a house exhibiting the Greek key pattern. Photo credit to the Wikimedia user ``Neoclassicism Enthusiast'' (CC BY-SA 4.0).~\cite{wiki:greek-key-image}}
	\label{figure:greek-key}
\end{figure}

Jensen in 2001, likely unaware of the newsgroup post, published an enumeration of ``maximally compact'' SAWs on a strip, also termed Hamiltonian paths on strips, with numerical estimates of their connective constants~\cite{jensen:compact-saws}. Later, in 2009, Nathaniel Johnston, likely unaware of Jensen's work, wrote a nice blog post on the topic~\cite{johnston:greek-key-blog}. In addition to the recurrence for $3 \times n$ Greek key tours, he points out that a conjecture for the $4 \times n$ Greek key tours can be found under OEIS sequence A046995. For $5 \times n$ Greek key tours, Johnston says ``It is now clear that this problem grows very large very quickly, and proceeding in this manner may not be (realistically) feasible.'' 

Let $g_{k,n}$ be the number of Greek key tours on the $k \times n$ grid graph and let $G_k(x) = \sum_n g_{k,n}x^n$. Using three small modifications of our techniques from previous sections, we can prove that $G_k(x)$ is rational for all $k$ and calculate $G_k(x)$ for $3 \leq k \leq 8$.

\textbf{Modification 1:} Since Greek key tours must visit every vertex, when constructing the directed graph to enumerate them and determining the valid extensions of each frame, we must only allow extensions that use every vertex in the new column. 

\textbf{Modification 2:} In previous sections we enumerated GSAWs on half-infinite grid graphs, and so in the construction of the directed graph for GSAWs, any accepting states were required to have a single segment and the last point in that segment could not be in the rightmost column (because then it would not be trapped, as the vertex to its right would still be open). Here we are counting Greek key tours on finite graphs, and therefore this condition on accepting states is not necessary. 

\textbf{Modification 3:} We are only interested in keeping track of the displacement of each tour, as its length can be determined from its displacement, and so in the directed graph we can omit the powers of $x$ in the weights of the transitions and keep only $y$.

With these two changes made, we can use the same procedure to build directed walks enumerating $g_{k,n}$ for fixed $k$ and then compute the corresponding generating functions.

\begin{theorem}
	For any $k \geq 1$, the generating function $G_k(x)$ that counts Greek key tours with height $k$ is rational.
\end{theorem}

The directed graphs we obtain are smaller than their counterparts from previous sections because many frames are no longer valid. The number of vertices in the minimized graphs for $k=3,\ldots,8$ is $8, 17, 37, 122, 263, 1065$.

\textbf{Height 3.} The generating function $G_3(x)$ for Greek key tours on height $3$ grid graphs is
\[
	G_3(x) = \frac{1-x-x^2+4x^3-x^4}{(1-2x)(1-2x^2)} = 1 + x + 3x^2 + 8x^3 + 17x^4 + 38x^5 + \cdots,
\]
confirming the claim in Womack's sci.math posting. This sequence is OEIS A046994. From the generating function we can derive that the exponential growth rate is $2$.

\textbf{Height 4.} For height 4 Greek key tours, we find
\begin{align*}
	G_4(x) &= \frac{1-2 x-2 x^{2}+11 x^{3}+4 x^{4}-10 x^{5}+4 x^{6}+2 x^{7}-3 x^{8}+x^{9}}{(1-x-3 x^{2}-x^{3}+x^{4})(1-2 x-2 x^{2}+2 x^{3}-x^{4})}\\
	&= 1+x +4 x^{2}+17 x^{3}+52 x^{4}+160 x^{5}+\ldots.
\end{align*}
confirming the conjectures of Dean Hickerson and Maksym Voznyy on OEIS sequence A046995. The exponential growth rate of the sequence is the reciprocal of the root of $1-2 x-2 x^{2}+2 x^{3}-x^{4}$ near $0.3939$, which is approximately $2.539$.

\textbf{Height 5.} For height 5 Greek key tours, we find
\begin{footnotesize}\begin{align*}
	G_5(x) &= \frac{1-3 x-13 x^{2}+58 x^{3}+25 x^{4}-162 x^{5}+63 x^{6}+113 x^{7}-61 x^{8}+12 x^{9}-11 x^{10}+25 x^{11}-17 x^{12}+5 x^{13}-3 x^{14}}{(1+x)(1-11 x^{2}-2 x^{6})(1-5 x+2 x^{2}+8 x^{3}-8 x^{4}+x^{5}-x^{6})}\\
	&= 1+x +5 x^{2}+38 x^{3}+160 x^{4}+824 x^{5}+\ldots
\end{align*}\end{footnotesize}{}%
confirming the conjecture of Colin Barker on OEIS sequence A145156. The exponential growth rate of the sequence is the reciprocal of the root of $1-5 x+2 x^{2}+8 x^{3}-8 x^{4}+x^{5}-x^{6}$ near $0.240$, which is approximately $4.167$. If $a(n)$ is the coefficient of $x^n$, then the linear recurrence for $a(n)$ is 
\begin{align*}
	a(n) &= 4 a(n - 1) + 14 a(n - 2) -54 a(n - 3) -33 a(n - 4) + 117 a(n - 5) + 2 a(n - 6)\\
	& \qquad  -84 a(n - 7) -6 a(n - 8) + 9 a(n - 9) -14 a(n - 11) -2 a(n - 13).
\end{align*}

\textbf{Height 6.} For height 6 Greek key tours, we find
\[
	G_6(x) = \frac{p(x)}{q_1(x)q_2(x)}
\]
where
\begin{align*}
	p(x) &= 1-6 x-34 x^{2}+264 x^{3}+366 x^{4}-3994 x^{5}+20 x^{6}+27007 x^{7}-19019 x^{8}-87139 x^{9}\\
			& \qquad +102484 x^{10}+130568 x^{11}-238632 x^{12}-50245 x^{13}+313267 x^{14}-117646 x^{15}\\
			& \qquad -284585 x^{16}+227721 x^{17}+227700 x^{18}-229923 x^{19}-186998 x^{20}+172233 x^{21}\\
			& \qquad +151742 x^{22}-131197 x^{23}-70563 x^{24}+87071 x^{25}+6034 x^{26}-35150 x^{27}+4755 x^{28}\\
			& \qquad +11858 x^{29}-2596 x^{30}-5759 x^{31}+1344 x^{32}+1351 x^{33}-425 x^{34}+14 x^{35}+47 x^{36}\\
			& \qquad -14 x^{37}+2 x^{38},\\
	q_1(x) &= 1-5 x-14 x^{2}+63 x^{3}-12 x^{4}-90 x^{5}+35 x^{6}+66 x^{7}-118 x^{8}+8 x^{9}+82 x^{10}-42 x^{11}\\
			& \qquad -28 x^{12}+4 x^{13}-2 x^{14}\\
	q_2(x) &= 1-2 x-29 x^{2}+25 x^{3}+237 x^{4}-156 x^{5}-794 x^{6}+330 x^{7}+1368 x^{8}-89 x^{9}-1325 x^{10}\\
			& \qquad -248 x^{11}+881 x^{12}+67 x^{13}-613 x^{14}-18 x^{15}+208 x^{16}+17 x^{17}-36 x^{18}+25 x^{19}\\
			& \qquad +35 x^{20}-3 x^{21}-4 x^{22}+x^{23}.
\end{align*}

This is sequence A160240 in the OEIS. The exponential growth rate of the sequence is the reciprocal of the root of $q_1(x)$ near $0.177$, which is approximately $5.652$. The linear recurrence for $a(n)$ involves terms from $a(n-1)$ to $a(n-37)$.

\textbf{Height 7.} For height 7 Greek key tours, we find
\[
	G_7(x) = \frac{p(x)}{q_1(x)q_2(x)}
\]
where $p$, $q_1$, and $q_2$ have degrees $72$, $35$, and $36$ respectively. The counting sequence is found in OEIS entry A160241. The exponential growth rate of the sequence is the reciprocal of the root of the denominator near $0.112$, which is approximately $8.909$. The linear recurrence for $a(n)$ involves terms from $a(n-1)$ to $a(n-71)$.

\textbf{Height 8.} For height 8 Greek key tours, we find
\[
	G_8(x) = \frac{p(x)}{q_1(x)q_2(x)}
\]
where $p$, $q_1$, and $q_2$ have degrees $203$, $66$, and $137$ respectively. We have added this to the OEIS as sequence A374307. The exponential growth rate of the sequence is the reciprocal of the root of the denominator near $ 0.0808$, which is approximately $ 12.382$. The linear recurrence for $a(n)$ involves terms from $a(n-1)$ to $a(n-203)$.

Data for the sizes of the directed graphs are found in Table 1. The growth rates of the number of vertices and edges in the minimized directed graphs are approximately $2.6^h$ and $1.4^h$ respectively.

\section{Concluding Remarks}
\label{section:conclusion}

\begin{table}
\centering
	\begin{tblr}{
		colspec = {l Q[r,m] Q[r,m] Q[r,m] Q[r,m] Q[r,m] Q[r,m] Q[r,m] Q[r,m]},
	    row{1} = {font=\bfseries, c},
	    cell{3-1}{1-1} = {halign=c, valign=b},
	    cell{1}{2} = {c=8}{c},
	}
		\toprule
		& Number of States in FSM Constructions \\
		\cmidrule{2-9}
			&
			\SetCell[c=2]{r}{non-probabilistic}
			&&			
			\SetCell[c=2]{r}{uniform model}
			&&			
			\SetCell[c=2]{r}{energetic model}
			&&			
			\SetCell[c=2]{r}{Greek key tours}\\
			$k$
			& $v$ & $e$ & $v$ & $e$ & $v$ & $e$ & $v$ & $e$ \\
		\midrule
		2	&	12 & 27 				& 12 & 27		& 54 & 123 			& 8 & 18		\\
			& 	6 & 11					& 7 & 14		& 24 & 59 			& 5 & 9		\\
		3	&	79 & 299 				& 79 & 299		& 1,078 & 3,996		& 32 & 96		\\
			& 	14 & 49					& 30 & 121 		& 225 & 965			& 8 & 22 \\
		4	&	706 & 4,537 			& 709 & 4,537	& 33,154 & 200,280 	& 213 & 848		\\
			&	41 & 304				& 162 & 1,163 	& 3,387 & 23,627 	& 17 & 78\\
		5	&	7201 & 78,408			& 7,201 & 78,408			&&		& 1,469 & 8,145		\\
			&	152 & 2,142				& 1,104 & 12,878 			&&		& 37 & 245 \\
		6	&	80,378 & 1,518,388		& 80,378 & 1,518,388		&&		& 14,696 & 120,562\\
			&	573 & 17,144			& 7,294 & 150,963 			&&		& 122 & 1,391\\
		7	&	954,791 & 31,488,152	& 954,791 & 31,488,152		&&		& 118,072 & 1,408,025\\
			& 	2,311 & 134,486			& 53,808 & 1,943,659 		&&		& 263 & 4,732\\
		8	&& 							&&	 						&&		& 1,446,249 & 26,229,823\\
			&&	 						&&	 						&&		& 1,065 & 31,627\\
		\bottomrule
	\end{tblr}
	\caption{The number of vertices and edges in the directed graphs constructed to enumerate various types of GSAWs. For each value of $k$, the upper row shows the numbers of vertices ($v$) and edges ($e$) before minimization, and the lower row shows these values after minimization.}
	\label{table:num-states}
\end{table}

We have partially extended the work of Klotz and Sullivan~\cite{klotz:gsaws1} by describing and implementing an automatic method to compute generating functions for GSAWs under the non-probabilistic model and two different probabilistic models. Table~\ref{table:num-states} shows the number of vertices and edges before and after minimizations for the combinatorial finite state machines that we have constructed here. These finite state machines have been used to exactly compute the mean trapping lengths and spans of growing self-avoiding walks in confined square lattices, summarized in Section 2. The exact trapping lengths in strips of heights 2 to 5 can be used to estimate an asymptotic value for trapping in the quarter-infinite plane that is comparable to that estimated from Monte Carlo simulations.

The quarter-infinite plane, with a mean trapping length of approximately 45, is a simpler case than the infinite square lattice with the famous mean trapping length of 71. While Klotz and Sullivan consider GSAWs on the grid graph $\GG_{\Z \times \{0, 1\}}$ that is infinite in both directions, we have only tackled half-infinite grid graphs $\GG_{\N \times \{0, \ldots, h-1\}}$. Our methods here can certainly be extended to the fully infinite case, but doing so would require an approach that builds frames in both the positive and negative directions simultaneously, or otherwise describes a method to ``glue together'' two partial GSAWs on half-infinite grid graphs in all valid ways. Klotz and Sullivan also derive results for other height-two graphs based on the triangular lattice, which we have not attempted here. We believe it is possible to use a more general approach than ours to derive the generating functions for GSAWs on many graphs that have a regular construction and whose height is constrained in some way.

\bibliographystyle{alpha}
\bibliography{paper.bib}

\appendix
\section{Monte Carlo Estimations}
\label{subsection:monte-carlo}

Our goal is to use results from exactly solvable restricted GSAWs to provide insights about trapping in the infinite-grid GSAW. Since our results pertain to walks beginning in the corner of a semi-infinite strip, they are most comparable to the GSAW beginning in the corner of a quarter-infinite square lattice. Results for trapping in such a lattice are not as well-established as those in an infinite square lattice, so we perform Monte Carlo simulations to determine what the asymptotic limits would be, were we able to compute the expected trapping length as the strip width tends to infinity.

Hemmer and Hemmer~\cite{hemmer:saw-average-71} use Monte Carlo simulation to predict that the expected length of a GSAW in the uniform model on the infinite square lattice (the grid graph $\GG_{\Z\times \Z}$) is $70.7 \pm 0.2$. In this section we report on our own Monte Carlo simulations for GSAWs in the uniform model on the square lattice, the half plane $\GG_{\N \times \Z}$, the quarter plane $\GG_{\N \times \N}$, as well as the half-infinite strips $\GG_h$ for $1 \leq h \leq 100$.

Our results in the main text prove that the expected length of a GSAW on $\GG_5$ in the uniform model is $\approx 26.51828485079$, and so we can use this result to test the accuracy of the confidence intervals produced by our Monte Carlo simulation. To do this, we repeated the following steps 10,000 times: generate one million GSAWs on $\GG_5$, use their lengths to compute confidence intervals for the expected length at the 50\%, 80\%, 90\%, 95\%, 99\%, and 99.99\% levels, and check whether the known true expected length is in each interval. We performed the same calculation for expected displacement as well. Table~\ref{table:conf-int-acc} records how often the true expected length and displacement were within the confidence intervals.

\begin{table}[h!]
\centering
	\begin{tblr}{
		colspec = {Q[1in,halign=c,valign=m] Q[2in,halign=c,valign=m] Q[2in,halign=c,valign=m]},
		row{1} = {font=\bfseries},
	    cell{1-2}{1} = {c=3}{c},
	}
		\toprule
		\SetCell{h,5in} Accuracy of Confidence Intervals for the Expected Length and Displacement of a GSAW on \bm{$\GG_5$} \\
		10,000 trials, each with 1,000,000 GSAWs\\
		\cmidrule{1-9}
			confidence level
			&
			\% of intervals containing the true expected length
			&
			\% of intervals containing the true expected displacement\\
		\midrule
		50\% & 50.43\% & 50.19\%\\
		80\% & 80.06\% & 79.99\%\\
		90\% & 90.09\% & 89.98\%\\
		95\% & 94.91\% & 95.18\%\\
		99\% & 99.02\% & 99.19\%\\
		99.99\% & 100.00\% & 100.00\%\\
		\bottomrule
	\end{tblr}
	\caption{Results of a Monte Carlo experiment to test the accuracy of confidence intervals at various levels against rigorously known results.}
	\label{table:conf-int-acc}
\end{table}

We then simulated 100 billion GSAWs each on the full plane, half plane, and quarter plane. For each graph, the starting point is $(0,0)$ and displacement is the difference between the maximum and minimum $x$-values reached by the walk. These results are found in Table~\ref{table:inf-gsaws}.

\begin{table}[h!]
\centering
	\begin{tblr}{
		colspec = {l Q[c,m] Q[c,m]},
	    row{1} = {font=\bfseries, c},
	    cell{1-2}{1} = {c=3}{c},
	}
		\toprule
		Expected Length and Displacement on Infinite Grid Graphs \\
		100 billion trials each, 99\% confidence intervals\\
		\midrule
			& expected length & expected displacement\\
		\midrule
		quarter plane & $45.3997248 \pm 0.0003198$ & $9.0996860 \pm 0.0000605$ \\
		half plane & $54.5329760 \pm 0.0003615$ & $11.9745362 \pm 0.0000660$ \\
		full plane & $70.7592964 \pm 0.0004116$ & $13.2641570 \pm 0.0000687$ \\
		\bottomrule
	\end{tblr}
	\caption{99\% confidence intervals for the expected length and displacement of GSAWs on several infinite grid graphs.}
	\label{table:inf-gsaws}
\end{table}

Finally, we simulated the expected length and displacement for each half-infinite strip $\GG_h$ with $1 \leq h \leq 100$. The values for $h \leq 5$ are known rigorously from the previous section. Table~\ref{table:strip-estimates} lists 99\% confidence intervals for $6 \leq h \leq 15$ and Figures~\ref{figure:length-graph} and \ref{figure:displacement-graph} show plots of the expected lengths and displacement against $1/h$. The horizontal lines in each graph show the estimated expected length and displacement for quarter plane walks.

\begin{table}[h!]
\centering
	\begin{tblr}{
		colspec = {l Q[c,m] Q[c,m]},
	    row{1} = {font=\bfseries, c},
	    row{4-7} = {font=\bfseries},
	    cell{1-2}{1} = {c=3}{c},
	}
		\toprule
		Expected Length and Displacement on $\bm{\GG_h}$ \\
		1 billion trials each, 99\% confidence intervals\\
		\midrule
			$h$ & expected length & expected displacement\\
		\midrule
		2 & 13 & 7 \\
		3 & 19.2127537600 & 7.7455012853 \\
		4 & 22.9750721653 & 7.8325670867 \\
		5 & 26.5182848508 & 8.0835861225 \\
		6 & $29.3528055 \pm 0.0025016$ & $8.2585576 \pm 0.0008403$\\
		7 & $31.7204199 \pm 0.0027267$ & $8.4039552 \pm 0.0008491$\\
		8 & $33.6971732 \pm 0.0029285$ & $8.5216546 \pm 0.0008569$\\
		9 & $35.3572211 \pm 0.0031094$ & $8.6167338 \pm 0.0008634$\\
		10 & $36.7626069 \pm 0.0032729$ & $8.6955012 \pm 0.0008692$\\
		11 & $37.9547103 \pm 0.0034203$ & $8.7599674 \pm 0.0008742$\\
		12 & $38.9747431 \pm 0.0035544$ & $8.8140591 \pm 0.0008788$\\
		13 & $39.8434124 \pm 0.0036753$ & $8.8583201 \pm 0.0008828$\\
		14 & $40.5888354 \pm 0.0037849$ & $8.8951760 \pm 0.0008862$\\
		15 & $41.2320529 \pm 0.0038847$ & $8.9268546 \pm 0.0008894$\\
		\bottomrule
	\end{tblr}
	\caption{99\% confidence intervals for the expected length and displacement of GSAWs on several infinite grid graphs. Rows for $2 \leq h \leq 5$, in bold are rigorous results from the previous sections, rounded to 10 decimal places.}
	\label{table:strip-estimates}
\end{table}

\begin{figure}
	\begin{center}
		\includegraphics[width=6in]{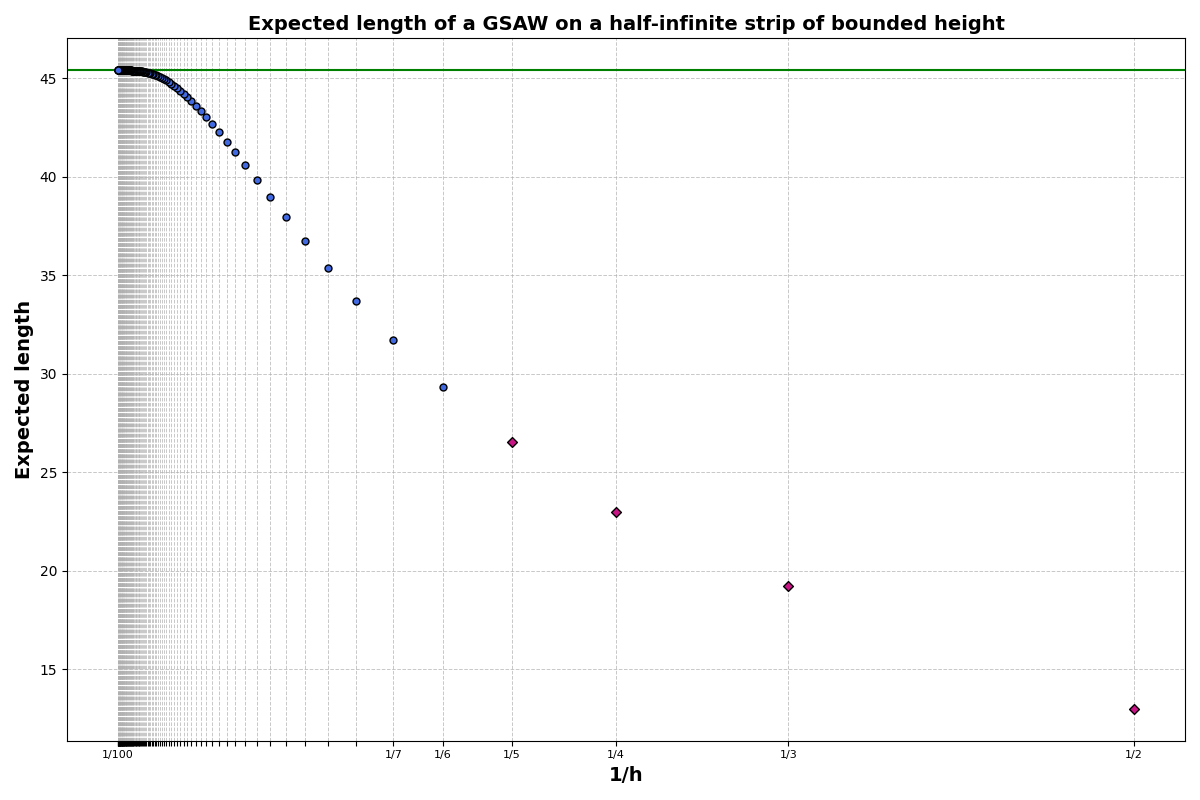}
		
		\includegraphics[width=6in]{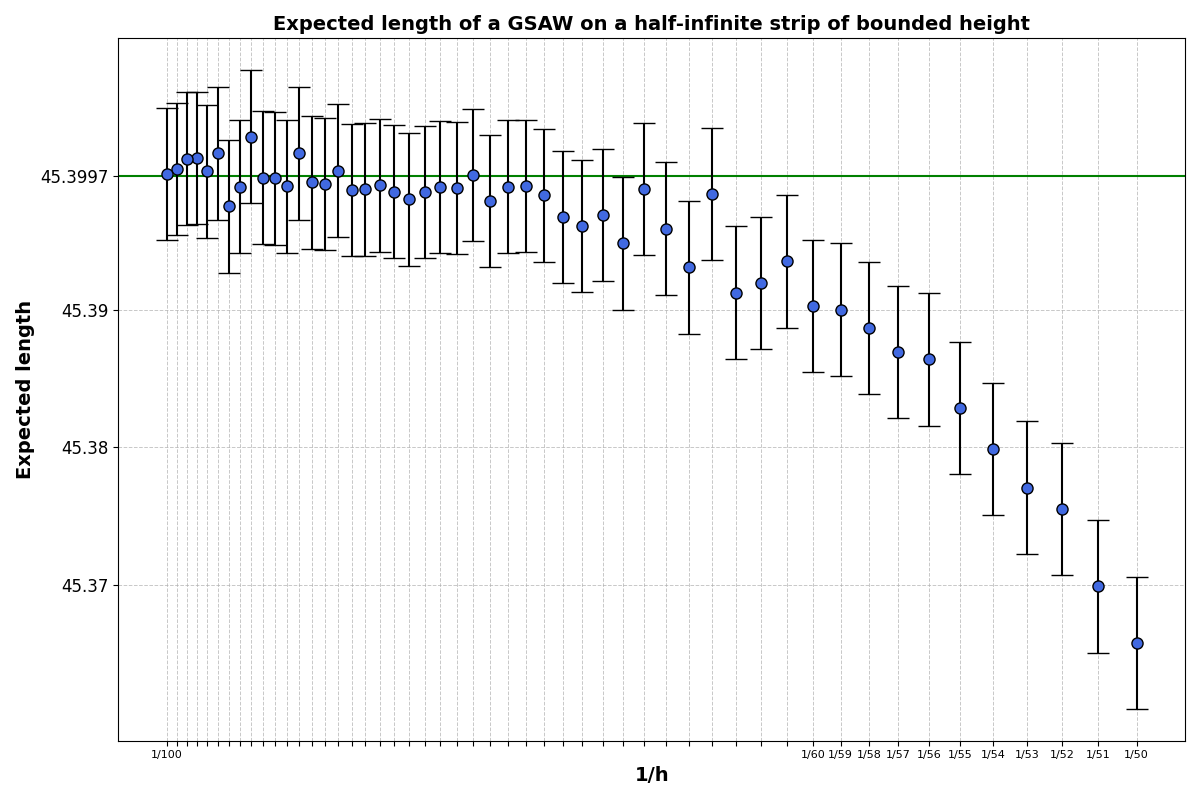}
	\end{center}
	\label{figure:length-graph}
	\caption{Monte Carlo estimates with 99\% confidence intervals for the expected length of a GSAW on a half-infinite strip of bounded height. Points plotted as a diamond are exact values determined in the previous sections. The bottom figure shows only $50 \leq h \leq 100$ with error bars corresponding to the 99\% confidence interval.}
\end{figure}

\begin{figure}
	\begin{center}
		\includegraphics[width=6in]{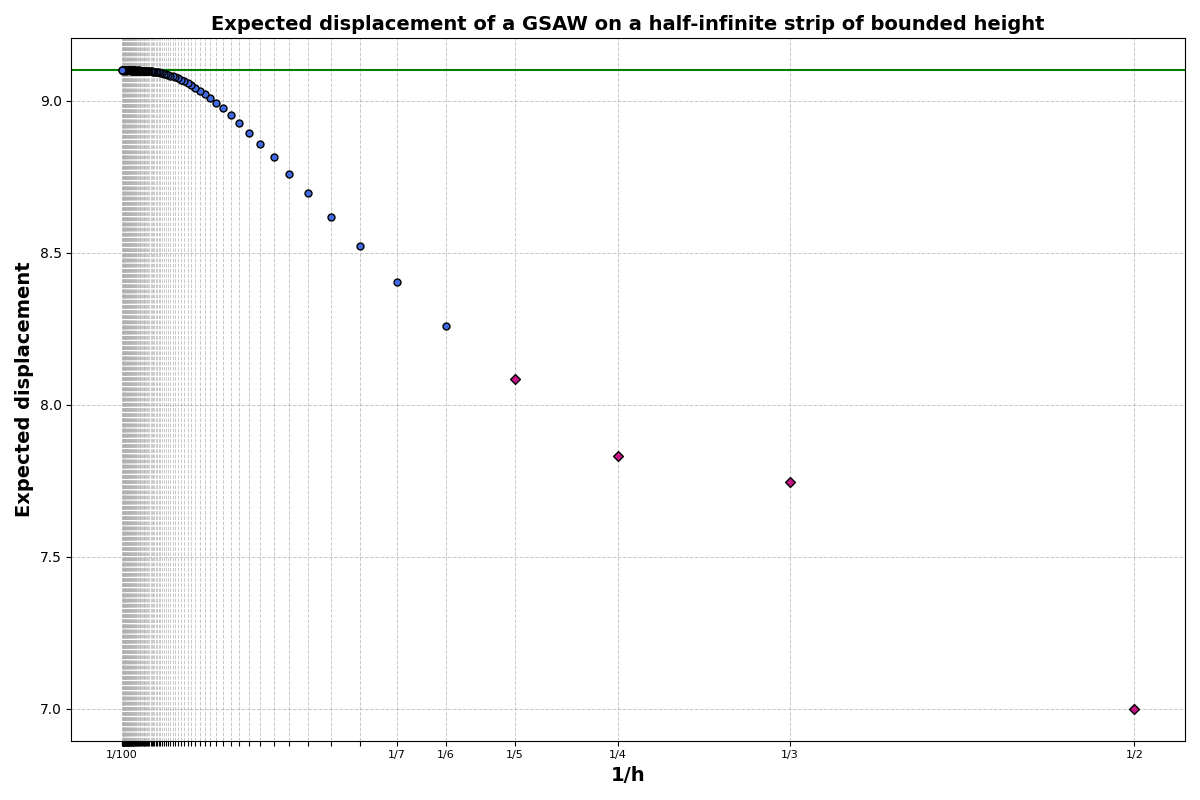}
		
		\includegraphics[width=6in]{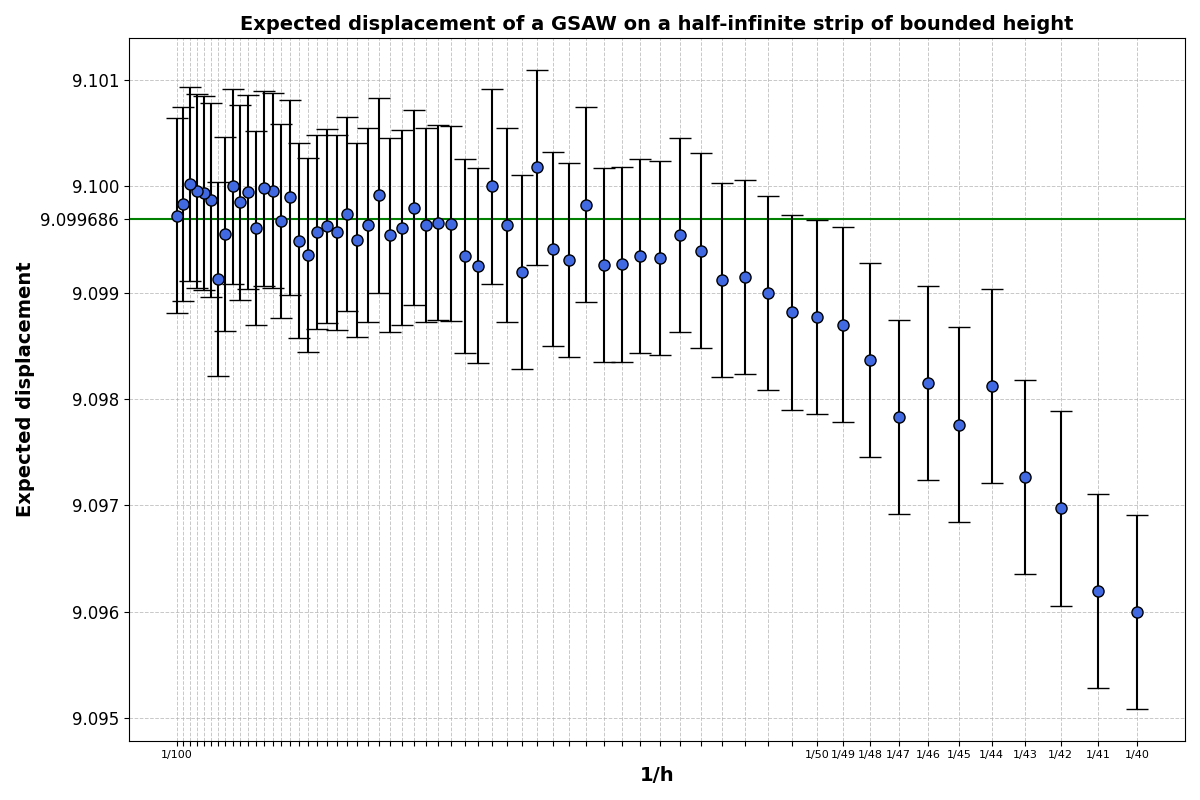}
	\end{center}
	\label{figure:displacement-graph}
	\caption{Monte Carlo estimates with 99\% confidence intervals for the expected displacement of a GSAW on a half-infinite strip of bounded height. Points plotted as a diamond are exact values determined in the previous sections. The bottom figure shows only $40 \leq h \leq 100$ with error bars corresponding to the 99\% confidence interval.}
\end{figure}

\end{document}